\documentclass[12pt]{amsart}
\usepackage{a4ams}
\usepackage{amsfonts}
\usepackage{amssymb}
\usepackage[arrow,matrix,curve]{xy}
\usepackage{ifthen}

\usepackage[active]{srcltx}

\makeindex

\newtheorem{theorem}{Theorem}[section]
\newtheorem{definition}[theorem]{Definition}

\newtheorem{lemma}[theorem]{Lemma}
\newtheorem{proposition}[theorem]{Proposition}
\newtheorem{corollary}[theorem]{Corollary}
\newtheorem{remark}[theorem]{Remark}

\newcommand{\N}{{\mathbb N}}

\newcommand{\Z}{{\mathbb Z}}
\newcommand{\F}{{\mathbb F}}
\newcommand{\E}{{\mathbb E}}
\newcommand{\C}{{\mathbb C}}

\def\cal{\mathcal}

\newcommand{\call}{{\cal L}}
\newcommand{\cale}{{\cal E}}

\newcommand{\calk}{{\cal K}}
\newcommand{\calf}{{\cal F}}
\newcommand{\calm}{{\cal M}}

\newcommand{\calh}{{\cal H}}

\newcommand{\calg}{{\cal G}}

\newcommand{\End}{ {\rm End} }

\newcommand{\Endo}{ {\rm End} }
\newcommand{\range}{{\mbox{range}}}

\begin{document}

\title[Descent homomorphism]{A descent homomorphism for semimultiplicative sets}
\author[B. Burgstaller]{Bernhard Burgstaller}
\address{Doppler Institute for mathematical physics, Trojanova 13, 12000 Prague, Czech Republic}
\email{bernhardburgstaller@yahoo.de}
\subjclass{19K35, 20N02, 46L55}
%
\thanks{The author was supported by Czech MEYS Grant LC06002.}

\begin{abstract}
We define and provide some basic analysis of various types of crossed products by semimultiplicative sets,
and then prove a $KK$-theoretical descent homomorphisms for semimultiplicative sets
in accord with the descent homomorphism for discrete groups.
%
%
%
\end{abstract}

\maketitle

\tableofcontents

\section{Introduction}

An associative semimultiplicative set is a set $G$ together with a partially defined associative multiplication.
For instance, categories, groupoids, semigroups, inverse semigroups and groups are associative semimultiplicative sets.
An equivariant $KK$-theory for semimultiplicative sets is defined in \cite{burgiSemimultiKK}, and in this theory the $G$-action is realized by linear
(non-adjointable) partial
isometries on $C^*$-algebras and Hilbert modules.
In this paper we prove a descent homomorphism for $KK^G$ and various types of crossed products,
$$KK^{H \times G} (A,B) \longrightarrow KK^H(A \rtimes G, B \rtimes G),$$
see Theorem \ref{theoremMain}, parallel to Kasparov's descent homomorphism for groups (\cite{kasparov1988}).
We consider four types of crossed products, the reduced one, the full one, the full strong one, and another one for so-called inversely
generated semigroups.


This work originated in an attempt
to generalise the Baum--Connes map for discrete groups (\cite{baumconneshigson1994}) to discrete semimultiplicative sets.
If $G$ is an inverse semigroup then this seems conceptually (and at least partially) to work, see \cite{burgiKKrDiscrete} and \cite{burgiGreenJulg}.
If $G$ is not an inverse semigroup then still certain reduced crossed products $A \rtimes_r G$ are isomorphic to inverse semigroup crossed products $A \rtimes S$, see Corollary \ref{corollaryIsomorphRedInvSemiCrossed}, and so for these crossed products one has potentitally a Baum--Connes theory.

In the full crossed product of a semimultiplicative set, however, one usually has non-commuting source and range projections of the underlying partial isometries, and this turns out to be an obstacle in constructing a Baum--Connes map similarly as for groups and groupoids:
these Baum--Connes maps can be constructed by a combination of a
descent homomorphism and an averaging map. Avaraging, however, fails for semimultiplicative sets
and their induced non-commuting projections on modules. (But even for inverse semigroups one cannot directly avarage but need to slice
modules at first (see \cite{burgiGreenJulg})).

Roughly speaking, the theory of crossed products by semimultiplicative sets is a theory of $C^*$-algebras
generated by partial isometries. Hence we generalise this point of view by considering also inversely generated semigroups,
which are $*$-semigroups that are generated by their invertible elements.

We give a brief overview of this paper. In Sections \ref{sectionSemimultiplicativeSets}-\ref{sectionPrelim}
we recall the basic definitions of equivariant $KK$-theory for semimultiplicative sets from \cite{burgiSemimultiKK}.
In Section \ref{sectionPartIso} we prove some facts about partial isometries in connection with $G$-actions.
Sections \ref{sectionActionsCrossed}-\ref{sectionReprL1} and Section \ref{sectionInvGenSemigroups}
are dedicated to the definition of the various crossed products; Section \ref{sectionInvGenSemigroups} also includes
the definition of equivariant $KK$-theory for inversely generated semigroups.
In Section \ref{sectionCompareWithGroupsKK} we compare semimultiplicative set $G$-equivariant $KK$-theory 
with Kasparov's $G$-equivariant $KK$-theory when $G$ is a group.
Sections \ref{sectionDescent}-\ref{sectionDescent3} occupy the proof of the descent homomorphism,
which is an adaption of Kasparov's proof in \cite{kasparov1988}.

\section{Semimultiplicative sets}    \label{sectionSemimultiplicativeSets}

\begin{definition}
{\rm A (general) {\em semimultiplicative set} $G$ is a set endowed
with a subset $G^{(2)} \subseteq G \times G$
and a map (written as a multiplication)
$$G^{(2)} \longrightarrow G: (s,t) \mapsto s t$$
satisfying the following weak associativity condition: $s(tu) = (st)u$ whenever both expressions are defined ($s,t,u \in G$).
}
\end{definition}

\begin{definition}
{\rm A semimultiplicative set $G$ is called {\em associative} if 
whenever
$(st)u$ or $s(tu)$ is defined, then both $(st)u$ and $s(tu)$ are defined
($s,t,u \in G$).}
\end{definition}


There is a similar notion called a semigroupoid (\cite{exelSemigroupoid}). A semigroupoid is an associative semimultiplicative set
with the property that $(s t) u$ is defined if and only if $st$ and $t u$ is defined. For instance, groupoids and small categories are semigroupoids. 
In general, however, an associative semimultiplicative set is not a semigroupoid, a
typical example being a ring $R$ without the zero element, so the semimultiplicative set $G = R \backslash \{0\}$ under the multiplication inherited from $R$. 
Examples for associative semimultiplicative sets include
groups, groupoids, small categories, inverse semigroups, semigroups, semigroupoids.
An associative semimultiplicative set is also called a partial semigroup in the literature
(see \cite{0809.04005}).


We remark that the weak associativity condition for a general semimultiplicative set is not essential in this paper. A general semimultiplicative set
is always realized by associative actions, so we require the weak associativity without essential loss of generality.
However, for instance, an arbitrary subset of a group is a general but not necessarily an associative semimultiplicative set.
Now the point is that general and associative semimultiplicative sets $G$ yield different classes of actions, since $G$ has to be realized by partial isometries.

If an associative semimultiplicative set $G$ has left cancellation, that is, for all $s,t_1,t_2 \in G$, $s t_1 = s t_2$ implies $t_1 = t_2$,
then we are able to define a left reduced $C^*$-algebra for $G$.
Write $(e_g)_{g \in G}$ for the canonical base in $\ell^2(G)$.

\begin{definition} \label{defLeftRegRep}
{\rm
Let $G$ be an associative semimultiplicative set with left cancellation.
The {\em left regular representation} of $G$
is the map $\lambda: G \longrightarrow B(\ell^2(G))$ given by
$$\lambda_g \Big (\sum_{h \in G} \alpha_h e_h \Big ) =
\sum_{h \in G, \, gh \mbox{ is defined}} \alpha_h e_{gh},$$
where $\alpha_h \in \C$.
The $C^*$-subalgebra of $B(\ell^2(G))$ generated by $\lambda(G)$ is called
the {\em reduced $C^*$-algebra} of $G$ and denoted by $C^*_r(G)$.
}
\end{definition}

\begin{definition}
{\rm
A {\em morphism} $\phi: G \longrightarrow H$ between two semimultiplicative sets $G$ and $H$
is a map satisfying $\phi(gh) = \phi(g) \phi(h)$ whenever $g h$ is defined ($g,h \in G$).
}
\end{definition}

\begin{definition}
{\rm
An {\em anti-morphism} $\varphi: G \longrightarrow H$ between semimultiplicative sets $G$ and $H$ is a map satisfying
$\varphi(g h)= \varphi(h) \varphi(g)$ whenever $gh$ is defined ($g,h \in G$).
}
\end{definition}

\begin{definition}
{\rm
A {\em left action} of a semimultiplicative set $G$
on a set $X$
is given by a subset
$Y \subseteq G \times X$
and a map
$$Y \longrightarrow X, \, (g,x) \mapsto g x$$
such that if $gh$ is defined, then
$(g h) x$ is defined if and only if
$g(h x)$ is defined, and in this case $(g h) x = g (h x)$
($g,h \in G, x \in X$).
}
\end{definition}

By the last definition we see that a $G$-action on a set is a morphism $\phi:G \longrightarrow {\rm PartFunc}(X)$ from $G$ into the set of partial functions on $X$.
(That is, if $g h$ is defined, then $\phi(gh)=\phi(g) \circ \phi(h)$ and the domain of both sides coincide.)
The domain of the composition of two partial functions is understood to be the maximal possible one.
The identity $\phi_1 = \phi_2$ of partial functions is understood to imply that both sides of the identity must have the same domain.

\begin{definition}
{\rm
A left $G$-action $\phi$ on $X$ is called {\em injective}
if the maps $\phi(g) \in {\rm PartFunc}(X)$ are injective on their domain for all $g \in G$.
}
\end{definition}

A {\em linear action} of $G$ on a vector space $X$ is a morphism $\phi: G \longrightarrow {\rm LinMap}(X)$ from $G$ into the linear maps on $X$.
The map $\lambda$ of Definition \ref{defLeftRegRep} may be checked to be a linear action on $\ell^2(G)$.
Left $G$-actions correspond to morphisms, and right $G$-actions to anti-morphisms.
That is, a right linear action on a vector space $X$ is an anti-morphism $\varphi: G \longrightarrow {\rm LinMap}(X)$.


\begin{definition} \label{Gaction_onHausdorffSet}
{\em
An injective left $G$-action $\phi$ on a Hausdorff space $X$ is {\em continuous} if
all maps $\phi(g) \in {\rm PartFunc}(X)$ are continuous and have clopen domains and ranges
for all $g \in G$.
}
\end{definition}

\section{$G$-Hilbert $C^*$-algebras and -modules}

\label{sectionPrelim}

In this section we recall the basic definitions for $G$-equivariant $KK$-theory for a general semimultiplicative set $G$ (\cite{burgiSemimultiKK}). 
All $C^*$-algebras and Hilbert modules are assumed to be
$\Z_2$-graded \cite{kasparov1981,kasparov1988}. If $\varepsilon$ is
a grading on a linear space $X$, then $\varepsilon(T)= \varepsilon T
\varepsilon$ is a grading on the space of linear maps $T$ on $X$.
All $*$-homomorphisms between $C^*$-algebras are supposed to respect
the grading.
We let $[x,y] = xy - (-1)^{\partial x \partial y } yx$ be the graded
commutator.

At first we shall define an action by a general semimultiplicative set $G$ on a $C^*$-algebra.
This is the next definition (from \cite{burgiSemimultiKK}, Definition 11, Definition 12, Definition 20, and the remark thereafter).

\begin{definition}    \label{definitionHilbertAlgebra}
{\rm
A {\em $G$-Hilbert $C^*$-algebra} $A$
is a $(\Z/2)$-graded $C^*$-algebra $A$ which is also regarded as a Hilbert module over itself
under the inner product $\langle x,y\rangle = x^* y$, and which is equipped with a
semimultiplicative set morphism
$$\alpha: G \longrightarrow \End(A)$$
and a semimultiplicative set anti-morphism
$$\alpha^*: G \longrightarrow \End(A)$$
such that $\alpha_g$ and $\alpha_g^*$ are zero-graded for all $g \in G$,
\begin{eqnarray*}
\alpha_g  &=& \alpha_g \alpha_g^* \alpha_g,\\
\alpha_g^*  &=& \alpha_g^* \alpha_g \alpha_g^*,
\end{eqnarray*}
and $\alpha_g^* \alpha_g$ and $\alpha_g \alpha_g^*$ are self-adjoint for all $g \in G$, and
\begin{eqnarray*}
\langle \alpha_g(x),y\rangle &=& \alpha_g (\langle x, \alpha_g^*(y) \rangle ), \\
\langle \alpha_g^*(x),y\rangle &=& \alpha_g^* (\langle x, \alpha_g(y)\rangle)
\end{eqnarray*}
holds for all $x,y \in A$ and all $g \in G$.
}
\end{definition}

We usually write simply $g(x)$ rather than $\alpha_g(x)$, and $g^*(x)$
rather than $\alpha_g^*(x)$.
Instead of $G$-Hilbert $C^*$-algebra we often say just Hilbert $C^*$-algebra if $G$ is clear from the context or unimportant.

\begin{definition}    \label{defGequiHilbertCstar}
{\rm
A {\em $G$-equivariant homomorphism} $\tau:A \rightarrow B$ between two Hilbert
$C^*$-algebras $A$ and $B$ is a $*$-homomorphism intertwining both the left and and the right $G$-action,
i.e.
$\tau(g(x)) = g(\tau(x))$ and 
$\tau(g^*(x)) = g^*(\tau(x))$
for all $x \in A$ and $g \in G$.
}
\end{definition}

\begin{definition}    \label{definitionGHilbertModule}
{\rm
A $G$-Hilbert module $\cale$ is a $(\Z/2)$-graded Hilbert module $\cale$ over a Hilbert $C^*$-algebra
$B$, such that $\cale$ is equipped with a semimultiplicative set
morphism
$$U: G \longrightarrow {\rm LinMap}(\cale)$$
and a semimultiplicative set anti-morphism
$$U^*: G \longrightarrow {\rm LinMap}(\cale)$$
such that $U_g$ and $U_g^*$ are zero-graded for all $g \in G$,
\begin{eqnarray*}
U_g  &=& U_g U_g^* U_g,\\
U_g^*  &=& U_g^* U_g U_g^*,
\end{eqnarray*}
and $U_g^* U_g$ and $U_g U_g^*$ are self-adjoint for all $g \in G$, and
\begin{eqnarray}
U_g(\xi b) &=& U_g(\xi) g(b),   \label{hilbmodrel1}\\
U_g^*(\xi b) &=& U_g^*(\xi) g^*(b),\\
\langle U_g(\xi),\eta \rangle &=& g ( \langle \xi, U_g^*(\eta) \rangle),  \label{hilbmodrel3} \\
\langle U_g^*(\xi), \eta\rangle &=& g^* (\langle \xi, U_g(\eta) \rangle)  \label{hilbmodrel4}
\end{eqnarray}
holds for all $\xi,\eta \in \cale, b \in B$ and $g \in G$.
}
\end{definition}

\begin{definition}  \label{definitionGstarEquivariance}
{\rm
Let $A$ and $B$ be $G$-Hilbert $C^*$-algebras and $\cale$ a $G$-Hilbert module over $B$.
A $*$-homomorphism
$\pi:A \longrightarrow \call(\cale)$
is called {\em $G$-equivariant}
if
\begin{eqnarray}
[U_g U_g^* ,\pi(a)] &=& 0,  \label{equivariantRep1}\\
{[U_g^* U_g ,\pi(a)] } &=& 0,  \label{equivariantRep2}\\
U_g \pi(a) U_g^* &=& \pi( g(a)) U_g U_g^*,  \label{equivariantRep3}\\
U_g^* \pi(a) U_g &=& \pi( g^*(a)) U_g^* U_g  \label{equivariantRep4}
\end{eqnarray}
for all $a \in A$ and $g \in G$.
}
\end{definition}

\begin{definition}
{\rm Let $A$ and $B$ be $G$-Hilbert $C^*$-algebras. A {\em
$G$-Hilbert $(A,B)$-bimodule} $\cale$ is a $G$-Hilbert $B$-module
$\cale$ together with a $G$-equivariant $*$-homomorphism $\pi:A
\longrightarrow \call(\cale)$. The homomorphism $\pi$ is often regarded as a left module multiplication of $A$ on $\cale$.}
\end{definition}

We also write $g(T)=U_g T U_g^*$ and $g^*(T) =
U_g^* T U_g$ for $g \in G$ and adjoint-able operators $T \in \call(\cale)$. Note that in general $\call(\cale)$ is not
a $G$-Hilbert $C^*$-algebra, as usually the action $g(\cdot)$ is not multiplicative, i.e. $g(TS) \neq g(T)
g(S)$. The {\em trivial} $G$-action on an object $X$
of a category is the action $\tau_g(x) = x$ for all $x \in X$ and  $g \in G$.

For a subset $C \subseteq \call(\cale)$ we set
\begin{eqnarray*}
Q_C &=& \{ T \in \call(\cale)|\, [T,c] \in \calk(\cale), \,\forall c \in C\},\\
I_C &=& \{ T \in \call(\cale)|\, c T \mbox{ and } T c \mbox{ are in } \calk(\cale),\, \forall c \in C\}.
\end{eqnarray*}
Here, $\calk(\cale)$ denotes the set of compact operators in the sense of
Kasparov (\cite{kasparov1988}).


\begin{definition}
{\rm
Let $A,B$ be $G$-Hilbert $C^*$-algebras.
Cycles in $\E^G(A,B)$ are Kasparov's cycles $(\pi,\cale,T)$ in $\E(A,B)$ (\cite{kasparov1988}) with the following addition:
$\cale$ is a $G$-Hilbert module
(Definition \ref{definitionGHilbertModule}) and $\pi: A \to \call(\cale)$ is a $G$-equivariant (Definition \ref{definitionGstarEquivariance}),
and the elements
\begin{equation}   \label{defCycleCondition}
g(T) - g (1)T, \,[g(1),T],\, [g^*(1),T]
\end{equation}
are in $I_A(\cale)$.
Parallel to Kasparov's theory, $KK^G(A,B)$ is defined to be
$\E^G(A,B)$ divided by homotopy induced by $\E^G(A,B[0,1])$.
}
\end{definition}

$KK^G(A,B)$ is functorial in $A$ and $B$ and allows an associative Kasparov product (\cite{burgiSemimultiKK}).

We recall that we have a diagonal $G$-action on tensor
products, see \cite[Lemmas 4 and 5]{burgiSemimultiKK}.
If $\cale_1$ and $\cale_2$ are $G$-Hilbert modules then $\cale_1 \otimes \cale_2$
is a $G$-Hilbert module, and $\cale_1 \otimes_{B_1} \cale_2$ is a $G$-Hilbert module
if $B_1 \longrightarrow \call(\cale_1)$ is a $G$-equivariant representation
(Definition \ref{definitionGstarEquivariance}),
both under the diagonal action $U^{(1)} \otimes U^{(2)}$.


\section{Partial isometries}

\label{sectionPartIso}



In this section we shall show that an action of a semimultiplicative set on a Hilbert module is realized by partial isometries
(Corollary \ref{GactionConsistentPartialIso}), where inverse elements go over to adjoint partial isometries
(Corollary \ref{corollaryHilbertactionInverses}).

A {\em projection} on a Hilbert module $\cale$ is a self-adjoint
idempotent map $P$ on $\cale$. Recall that the identity $P(\cale) =
\calh$ links
complemented subspaces $\calh$ of $\cale$ with projections $P$ on
$\cale$ in a bijective way.

\begin{definition} \label{defintionPartialisometry}
{\rm A {\em partial isometry} $T$ on a Hilbert-module $\cale$ is a
linear map $T: \cale \rightarrow \cale$ for which there exist two
complemented subspaces $\calh_0$ and $\calh_1$ in $\cale$ such that
$T$ maps $\calh_0$ norm-isometrically onto $\calh_1$ and vanishes on
$\calh_0^{\bot}$. }
\end{definition}

Notice that we do not require that a partial isometry $T$ is adjoint-able.
(For instance, in Lance's book \cite{lance}, partial isometries are supposed to be adjoint-able.)
The projections $Q$ and $P$ of a partial isometry $T$
as in Definition \ref{defintionPartialisometry}
projecting onto $\calh_0$ and $\calh_1$,
respectively, are called the {\em source} and {\em
range projections} of $T$. Since $\calh_0^{\bot} = \ker(T)$ and
$\calh_1 = \range(T)$, $Q$ and $P$ are
uniquely determined by $T$. The {\em inverse partial isometry} $S$ of  $T$,
also denoted by $S=T^*$, is the unique partial isometry $S$ on
$\cale$
which vanishes on $\calh_1^\bot$ and satisfies $S|_{\calh_1} =
(T|_{\calh_0})^{-1}$. If $T$ happens to be adjoint-able then the
notation $T^*$ cannot cause confusion as in this case the inverse
partial isometry is the adjoint of $T$,
see \cite{lance}.
The set of partial isometries of $\cale$ is denoted by ${\rm PartIso}(\cale)$.

\begin{lemma} \label{characterizationlemmaPartIso}
$T$ is a partial isometry if and only if $T$ is a norm contractive linear map
and there exists a norm contractive linear map
$S:\cale \rightarrow \cale$
such that $S T$ and $T S$
are projections, $T = T S T$ and $S = S T S$.
In this case $S = T^*$.
\end{lemma}

\begin{proof}
Since $S$ and $T$ are contractive, we have $\| T x \| = \|TST x\| \le \|STx\| \le \| T x \|$
and $\|S y \| = \| T S y\|$ for all $x,y \in \cale$.
Thus $T$ is a partial isometry with source and range projections
$ST$ and $TS$, respectively, and $S=T^*$.
\end{proof}

\begin{corollary} \label{GactionConsistentPartialIso}
If $U$ is a $G$-action on a Hilbert module then $U_g$ is a partial isometry with
inverse partial isometry $U_g^*$ ($g \in G$).
\end{corollary}

\begin{proof}
The boundedness of $U_g$ follows from $\|\langle U_g x , U_g x\rangle\| =\|g(\langle x , U_g^* U_g x \rangle)\| \le \|x\|^2$,
and then one applies Lemma \ref{GactionConsistentPartialIso}.
\end{proof}

\begin{lemma} \label{IdemLemma}
A partial isometry $T$ satisfying $T=T T$ and $T^* = T^* T^*$ is a projection.
\end{lemma}

\begin{proof}
Let $x \in \cale$. Set $y = T x$. Then $T y= T T y = Tx = y$.
Let $y= y_0 + y_1$ with $y_0 = T^* Ty$ and $y_1 = (1- T^* T) y$ be the orthogonal decomposition.
Then $T^* y = T^* T y = y_0$. Hence, $y_0 = T^* y= T^* T^* y = T^* y_0$, and thus
$T^* (y_0 + y_1) = y_0 = T^* y_0$, and so $T^* y_1= 0$.
We thus have
\begin{eqnarray*}
0 &=& \langle T T^* y_1, y_0 \rangle = \langle y_1, T T^* y_0 \rangle
= \langle y_1, T y_0 \rangle = \langle y_1, T T^* T y \rangle\\
&=& \langle y_1, T y \rangle = \langle y_1, y \rangle = \langle y_1, y_1 \rangle .
\end{eqnarray*}
Thus $y_1=0$ and so $T^* T y = y_0 = y = T y$. Hence, $T^* T T x = T T x$, and so $T^* T x = Tx$. Since $x$ was arbitrary,
$T^* T = T$, and thus $T$ is a projection.
\end{proof}

\begin{definition}
{\rm An element $g$ of a semimultiplicative set $G$ is called {\em
invertible} if there exists an element $h \in G$ such that $ghg=g$
and $h g h =h$. }
\end{definition}

Even if the inverse element $h$ may not be unique, we occasionally denote a given choice by
$h=g^{-1}$.

\begin{corollary} \label{corollaryHilbertactionInverses}
Assume that $\cale$ is a $G$-Hilbert module and $g \in G$ is invertible.
Then $U_g^*= U_{g^{-1}}$ and $U_{g^{-1}}^*= U_g$.
\end{corollary}

\begin{proof}
Set $T= U_{g g^{-1}} = U_g U_{g^{-1}}$. Then $T T= T$ and $T^* T^* = T^*$. Hence $T$ is a projection by 
Lemma \ref{IdemLemma}. Similarly, $U_{g^{-1}} U_g$ is a projection.
By Lemma \ref{characterizationlemmaPartIso} (for $S:= U_g$ and $T:= U_{g^{-1}}$), $U_g^* =U_{g^{-1}}$.
\end{proof}




\section{Algebraic crossed products}

\label{sectionActionsCrossed}


In this section $G$ denotes a discrete general semimultiplicative set (if nothing else is said).
For the work with crossed products we shall need to consider also free products of elements of $G$ and their adjoints, and for that purpose
we shall introduce $G^*$ below.  

\begin{definition}
{\rm
An {\em involution} on a semigroup $S$ is a map $*: S \longrightarrow S: s \mapsto s^*$ such that
${(s^*)}^* = s$ and $(s t)^* = t^* s^*$ for all $s,t \in S$.
}
\end{definition}

\begin{definition}
{\rm
Define $F(G)$ to be the free semigroup generated by two copies of $G$. The elements of the second copy of $G$ are denoted by $g^*$ for $g \in G$
and stand for adjoint elements.
In other words, elements $\gamma$ of $F(G)$ consist of formal words
$\gamma = x_1^{\epsilon_1} \ldots x_n^{\epsilon_n}$
with $x_i \in G$ and $\epsilon_i \in \{1,*\}$.
}
\end{definition}

We shall occasionally denote the multiplication in $G$ by $g \odot h$ ($g,h \in G)$ to distinguish it
from the multiplication in $F(G)$.


\begin{definition} \label{simple_equivalences}
{\rm
Define $G^*$ to be the semigroup which is the quotient semigroup
of $F(G)$ by the following {\em elementary equivalences} defined for all $g,h \in G$.
%
%
\begin{eqnarray*}
&&g \odot h = g h, \quad (g \odot h)^* = h^* g^* \qquad \mbox{if $g \odot h$ is defined}\\
&&g = g g^* g, \quad g^* = g^* g g^*.
\end{eqnarray*}
}
\end{definition}

In other words, elements of $G^*$ consist of representatives living in $F(G)$,
and two representatives $\gamma, \delta \in F(G)$ are equivalent, if there is a finite sequence
of representatives in $F(G)$ starting with $\gamma$ and ending with $\delta$,
where two representatives in this sequence differ only by a single elementary equivalence
(within a word).

$G^*$ is an involutive semigroup by concatenation and taking the formal adjoints of representatives
of $F(G)$. For simplicity we shall omit the class brackets and write $g$ rather than the class $[g]$ for elements in $G^*$, where
$g \in F(G)$ is a representative.
Note that an element in $G^*$ need not be invertible: if $g,h \in G$ are incomposable in $G$ then
usually $g h (g h)^* g h \neq g h$ in $G^*$.

\begin{lemma} \label{lemmafromMorphtoStarMorph}
A morphism (resp. anti-morphism) $\varphi:G \longrightarrow H$ between semimultiplicative sets $G$ and $H$
extends canonically to a $*$-morphism (resp. $*$-anti-morphism) $G^* \longrightarrow H^*$.
\end{lemma}

\begin{proof}
A morphism $\varphi: G \rightarrow H$ induces a canonical $*$-morphism $F(G) \longrightarrow F(H)$
which respects the elementary equivalences of Definition \ref{simple_equivalences}.
\end{proof}

%


%

For the work with crossed products it is useful to extend a $G$-action to a $G^*$-action, and this is what the next couple of lemmas
will be about.

\begin{lemma} \label{inverse_Action}
If $\phi$ is an injective $G$-action on a set $X$ and
$g \in G$ is invertible in $G$ then $\phi(g)^{-1} = \phi(g^{-1})$.
\end{lemma}

\begin{proof}
Let $h$ be an inverse element for $g$.
If $g x$ is defined then $(ghg) x= g (h (gx))$ is defined,
so $h (gx)$ is defined; and conversely, if $hx=hghx$ is defined then $x=ghx$ by injectivity of the $G$-action.
We have checked that the range of $\phi(g)$ is the domain of $\phi(h)$.
From $g hg x= gx$ it follows $gh x = x$ by injectivity of the $G$-action, and similarly $hgx=x$.
Thus $\phi(g)$ and $\phi(h)$ are inverses to each other.
\end{proof}

\begin{lemma} \label{lemmaActionGstar}
A continuous injective left $G$-action on a Hausdorff space $X$ can be extended to a continuous injective left $G^*$-action on $X$.
\end{lemma}

\begin{proof}
Let $\phi: G \longrightarrow {\rm PartFunc}(X)$ be the $G$-action on $X$.
For $g=g_1^{\epsilon_1} \ldots g_n^{\epsilon_n} \in F(G)$ ($g_i \in G, \epsilon_i \in \{1,*\}$) define
\begin{equation}   \label{extendphi}
\hat \phi(g) = \phi({g_1})^{\epsilon_1} \circ \ldots \circ \phi({g_n})^{\epsilon_n}.
\end{equation}
Here, $\phi(g)^*$ denotes the inverse partial function for $\phi(g)$.
We have to show that (\ref{extendphi}) factors through $G^*$, in other words, we must show that $\phi$ is invariant under the elementary equivalences of
Definition \ref{simple_equivalences}.

Let $s, t \in F(G)$, $g,h \in G$ and $g \odot h \in G$ be defined. Then $s (g \odot h)^* t = s h^* g^* t$ in $G^*$.
By (\ref{extendphi}) and the definition of an action $\phi$ we have
\begin{eqnarray*}
&&\hat \phi(s (g \odot h)^* t)= \phi(s) \big (\phi(g \odot h) \big)^* \phi(t)\\
&=& \phi(s) \big(\phi(g) \phi(h) \big)^* \phi(t) = \phi(s) \phi(h)^* \phi(g)^* \phi(t) = \hat \phi(s h^* g^* t).
\end{eqnarray*}
The other elementary equivalences are checked similarly.
It is easy to see that the extended $\phi$ is also a continuous action
(the inverse partial functions and composition of partial functions have clopen domains and ranges again). 
%
%
%
\end{proof}

\begin{lemma} \label{lemmaHilbertmoduleGstar}
Every $G$-Hilbert $B$-module $\cale$ induces a morphism $\hat U:G^*
\longrightarrow {\rm LinMap}(\cale)$ extending the $G$-action $U$ on
$\cale$.
The relations (\ref{hilbmodrel1})-(\ref{hilbmodrel4}) hold also for all $g\in G^*$.
\end{lemma}

\begin{proof}
For $g_1^{\epsilon_1} \ldots g_n^{\epsilon_n} \in F(G)$ ($g_i \in G, \epsilon_i \in \{1,*\}$) define
$$\hat U_{g_1^{\epsilon_1} \ldots g_n^{\epsilon_n}} = U_{g_1}^{\epsilon_1} \ldots U_{g_n}^{\epsilon_n}.$$
This map respects the elementary equivalences of Definition \ref{simple_equivalences} since $U$ and $U^*$
are a morphism and anti-morphism, respectively, by Definition \ref{definitionGHilbertModule}.
Consequently $\hat U$ factors through $G^*$.
The relations (\ref{hilbmodrel1})-(\ref{hilbmodrel4}) are checked by induction
(recall \cite[Lemma 3]{burgiSemimultiKK}).
%
%
\end{proof}

We emphasize that $\hat U$ of the last lemma is a morphism but not a $*$-morphism.
Usually $\cale$ is not a $G^*$-Hilbert module as $\hat U_g$ need not to be a partial isometry for $g \in G^*$.
It may thus be suggestive to write $\hat U_g^*$ for $U_{g^*}$ ($g \in G^*$) but one should be aware
that this star might not be a (well defined) operator on the sets of $U_g$'s.
There is no (obvious) involution in the image of $\hat U$.

We shall usually write $U$ rather than $\hat U$.



\begin{lemma} \label{lemmaHilbertCstarGstar}
(i) Every $G$-Hilbert $C^*$-algebra $A$ is also a $G^*$-Hilbert $C^*$-algebra.
In particular, there is a $*$-morphism  $\hat \alpha: G^*  \longrightarrow {\rm PartIso}(A) \cap \Endo(A)$ extending the $G$-action $\alpha$.

(ii) Every $G$-equivariant representation $\pi: A \longrightarrow \call(\cale)$ of $A$ on a $G$-Hilbert module $\cale$ is $G^*$-equivariant
in the sense that the identities (\ref{equivariantRep1})-(\ref{equivariantRep4}) hold also for $g \in G^*$
(where $U_g^*$ has to be interpreted as $U_{g^*}$).

(iii) For all $a,b \in A$ and $g \in G^*$ one has
$gg^{*}(ab) = gg^{*}(a) b = a gg^{*}(b)$.
\end{lemma}

\begin{proof}
We extend the $G$-action $\alpha$ to a morphism $\hat \alpha$ on $A$ according to
Lemma \ref{lemmaHilbertmoduleGstar}. Let $g,h \in G^*$ and $a,b \in
A$.
We may write $\alpha_g \alpha_g^*(a) b = \langle \hat \alpha_g \hat \alpha_g^*(a^*), b \rangle$ for all $a \in A$ and $g \in G^*$.
Writing $\hat \alpha_g (a) = g(a)$, by identity (\ref{equivariantRep3}) (Lemma \ref{lemmaHilbertmoduleGstar}) we have
$$g g^* (a) b = \langle g g^* (a^*) , b  \rangle = g (
g^* (a) g^*(b)) = g g^*(a) g g^*(b),$$
and similarly $a g
g^*( b) = g g^*(a) g g^*( b)$. Hence $g g^*(a) b = a g
g^*(b)$,
that is, $g g^* \equiv \hat \alpha_g \hat \alpha_g^*$ is self-adjoint. Since $g g^* g g^*(a) b =
g g^* (a) g g^* (b) = g g^*(a) b$, $g g^*$ is a
projection. These identites prove already (iii). Now
$$g g^* h h^* (a) b = g g^* (h h^*(a) b)
= g g^* (a h h^* (b)) = g g^* (a) h h^* (b) =  h h^*
g g^* (a) b,$$
that is, $g g^*$ and $h h^*$ commute. Hence $g \equiv \hat \alpha_g$
is the product of partial isometries $\alpha_i,\alpha_j^*$ ($i,j \in G$) with commuting range and source
projections and thus by a standard inductive proof and Lemma
\ref{characterizationlemmaPartIso} a partial isometry with inverse
partial isometry $\hat \alpha_{g}^* =
\hat \alpha_{g^*}$.
This shows that $\hat \alpha$ maps into the partial isometries, and is thus a $G^*$-action, which proves (i).
The $G^*$-equivariance claimed in (ii) (meaning that the formulas of
Definion \ref{definitionGstarEquivariance} hold) follows by
induction; see also \cite[Lemma 9]{burgiSemimultiKK}.
\end{proof}

%

\begin{lemma} \label{lemmaGspaceHilbertCstar}
Let $X$ be a Hausdorff space equipped with an injective continuous right $G$-action $\tau$.
Then $C_0(X)$ is a $G$-Hilbert $C^*$-algebra under the action 
$\alpha_g(f)x =1_{\{\tau_g(x) \mbox{ is defined}\} } f(\tau_g(x))$ ($\alpha_g^* := \alpha_g^{-1}$) for $f \in C_0(X), g \in G$ and $x \in X$.
%
\end{lemma}

\begin{proof}
By definition of a continuous action $\tau$ on $X$, the domain and range, respectively, of $\tau_g$ is a clopen subset $D_g$ and $R_g$, respectively, of $X$.
So $\alpha_g(f)$ is indeed a continuous function. $\alpha_g$ projects onto $1_{D_g} C_0(X)$, and $\alpha_g$ moves $1_{R_g} C_0(X)$ onto $1_{D_g} C_0(X)$.
$\alpha_g^*$ is the inverse map.
It is straightforward to verify Definition \ref{definitionHilbertAlgebra} and this is left to the reader.
%
%
\end{proof}


%

We give another characterization of a Hilbert $C^*$-algebra.

\begin{lemma}   \label{lemmaCharacterizationHilbertCstar}
Let $A$ be a $C^*$-algebra. Then $A$ is a Hilbert $C^*$-algebra with
$G$-action $\alpha$ if and only if $\alpha$ is a morphism $\alpha:G \longrightarrow {\rm PartIso}(A) \cap {\rm End}(A)$, and
for every $g \in G$ the source and range projections $\alpha_g^*
\alpha_g, \alpha_g \alpha_g^*$ are in $Z \calm(A)$ (center of the
multiplier algebra of $A$).
\end{lemma}

\begin{proof}
If $A$ is a Hilbert $C^*$-algebra then source and range projections
of $\alpha_g$ are in $Z\calm(A)$ as remarked in \cite[Section
7]{burgiSemimultiKK}. Conversely, assume the condition. Then $A
\subseteq \call(A)$ by left multiplication. Since $g g^*$ is in
$Z \calm (A)$, $g g^*$ commutes with the left multiplication
operator $L_a(b)= ab$ ($a,b\in A$), and so $g g^*(a b ) = a g
g^*(b)$. Moreover, $g g^*(ab) = g g^* (a) b$ (since $g g^* \in \call(\cale)$). In particular,
$gg^* (a) b = g g^*
(a b) = a g g^*(b)$. With this one easily gets $\langle
g(a),b\rangle = g \langle a, g^*(b) \rangle$.
\end{proof}

We shall now come to crossed products by $G$.

\begin{definition}   \label{definitionAlgCrossedProduct}
{\rm Let $A$ be a $G$-Hilbert $C^*$-algebra.
Write $\F(G,A)$
for the universal $*$-algebra
generated by $A$ and $G$ subject to the following relations: The
$*$-algebraic relations of $A$ are respected
and the identities
\begin{eqnarray}
&&g \odot h = g h \qquad \mbox{if $g \odot h$ is defined}, \label{ident0}\\
&&g g^* g = g, \quad
g g^* a = a g g^*,\quad  \label{ident1}
g^* g a = a g^* g,\\
&&g a g^* = g(a) g g^*, \quad  \label{ident2}
g^* a g = g^*(a) g^* g
\end{eqnarray}
hold true for
all $g,h \in G$ and $a \in A$.
}
\end{definition}

\begin{definition}
{\rm
Let $A$ be a $G$-Hilbert $C^*$-algebra.
The {\em algebraic crossed product}
$A \rtimes_{\rm alg} G$ of $A$ by
$G$ is the $*$-subalgebra of $\F(G,A)$
generated by the set
$$\{a g \in\F(G,A)|\, a \in A, g \in G\}.$$
}
\end{definition}




Let $A$ be a $G$-Hilbert $C^*$-algebra.
Write
$$A_g = g g^*(A)$$
for $g \in G^{*}$.
$A_g$ is a two-sided closed ideal in $A$
by Lemma \ref{lemmaHilbertCstarGstar} (iii).

\begin{lemma} \label{lemmaReprConvolutionAlg}
$A \rtimes_{\rm alg} G$ is canonically isomorphic to the $*$-algebra $C_c(G^*,A)$ consisting of
formal finite sums $\sum_{g \in G^*} a_g g$ ($a_g \in A_g$)
with involution
$$\Big (\sum_{g \in G^*} a_g g \Big)^* = \sum_{g \in G^*} g^*(a_g^*) g^*$$
and convolution product
$$\sum_{g \in G^*} a_g g \sum_{h \in G^*} b_h h = \sum_{g,h \in G^*} a_g g(b_h) g h.$$
\end{lemma}

\begin{proof}
By induction on the length of a word in $G^*$ one checks that $g a = g(a) g$ holds in $\F(G,A)$ for all $g \in G^*$.
Note that $g(a) = g g^* g(a) \in A_g$ since the $G^*$-action on a Hilbert $C^*$-algebra is realized by partial isometries (Lemma \ref{lemmaHilbertCstarGstar}).
One has
\begin{equation}   \label{ginfa}
a g = (g^* a^*)^* = (g^*(a^*) g^*)^* = g g^*(a) g = a_g g
\end{equation}
for all $a \in A$ and $g \in G^*$, where $a_g := g g^*(a) \in A_g$.
It follows that
\begin{eqnarray}
&& g g^* a = g g^*(a) g g^* = a g g^*  \label{someeq123}\\
&& g a g^* = g(a) g g^*  \label{someeq124}
\end{eqnarray}
for all $a \in A$ and $g \in G^*$.
Define $D=A \oplus C_c(G^{*}, A) \oplus G^*$.
Endow $D$ with the algebraic structure on the summands as given, and between the summands as we have it in $\F(G,A)$, for instance $g \cdot a = g(a) g \in C_c(G^*,A)$
for $a \in A$ and $g \in G^*$.
By universality of $\F(G,A)$ there is a $*$-homomorphism
$\phi: \F(G,A) \longrightarrow D$ such that $\phi(a) = a$ and $\phi(g) = g$ for all $a \in A$ and $g \in G^*$
(using (\ref{someeq123})-(\ref{someeq124})).
It is obviously injective, as $D$, and particularly $C_c(G^*,A)$, is a direct sum.
The restriction $\phi'$ of $\phi$ to $A \rtimes_{\rm alg} G$ yields $C_c(G^*,A)$.
The surjectivity of $\phi'$ follows by induction from the factorization
$$a g h = (a^{1/2} g) (g^* (a^{1/2}) h)$$
for $a \in A_+$ and $g,h \in G$.
\end{proof}

\begin{lemma}   \label{lemmaIsomFGA}
(i) There is a linear isomorphism
$$\F(G,A) \cong A \oplus C_c(G^{*}, A) \oplus G^*.$$
(ii) The identities (\ref{ident1})-(\ref{ident2}) hold for all $a \in A$ and $g \in G^*$.
\end{lemma}

\begin{proof}
This was proved in Lemma \ref{lemmaReprConvolutionAlg}.
\end{proof}

One usually has not cancellation in $G^*$, even if $G$ has it.
Assume for instance that $g,h \in G$
are not invertible and not composable in $G$. Then usually $h \neq g^* g h$ in $G^*$.
For this reason we need not have a transformation like `$x = gh$ $\Leftrightarrow$ $g^* x = h$' in the convolution
product of Lemma \ref{lemmaReprConvolutionAlg}.

\begin{definition}   \label{defCovHil}
{\rm
By a {\em covariant representation} of a
$G$-Hilbert $C^*$-algebra $A$ we mean a $G$-equivariant
representation $\pi:A \longrightarrow B(H)$ on a $G$-Hilbert space $H$ (Definition \ref{definitionGHilbertModule} with
trivial $G$-action on $\C$) in the sense of Definition
\ref{definitionGstarEquivariance}.
}
\end{definition}

\begin{lemma} \label{lemmaBijectionCovRepF}
Restricting a $*$-homomorphism $\phi: \F(G,A) \longrightarrow B(H)$ of $\F(G,A)$ to $A$ and $G$ gives a covariant representation
$(\phi|_A, \phi|_G,H)$
of $A$.
Conversely, a covariant representation $(\pi,u,H)$ of $A$ extends canonically to a representation $\phi:\F(G,A) \longrightarrow B(H)$ of $\F(G,A)$
determined by $\phi|_A= \pi$ and $\phi|_G = u$.
This correspondence between representations of $\F(G,A)$ and covariant representations of $A$ is a bijection.
\end{lemma}

By the last lemma it is often comfortable to work with {\em one} homomorphism $\phi$ rather than an equivariant representation.
A covariant representation of $A \rtimes_{\rm alg} G$ is then just a restriction of $\phi$. We have the following diagram
(where $\iota$ denotes the canonical embedding).
 
$$\begin{xy}
\xymatrix{
\F(G,A) \ar[drr]^{\phi}   & \\
A \rtimes_{\rm alg} G  \ar[u]_{\iota}  \ar[rr]^{(\phi|_A,\phi|_G,H)}_{\phi|_{A \rtimes_{\rm alg} G}} & & B(H)
}
\end{xy}$$





\section{Full crossed products}

%

\begin{definition}
{\rm
Let $(\pi,u,H)$ be a $G$-covariant representation of a $G$-Hilbert $C^*$-algebra $A$
and $\phi$ its induced representation on $\F(G,A)$.
The $C^*$-algebra $A \rtimes_{(\pi,u,H)} G$ {\em induced by this covariant representation} is the norm closure
of $\phi (A \rtimes_{\rm alg} G)$.
}
\end{definition}


\begin{definition}
{\rm
The {\em universal covariant representation} of $A$ is the direct sum of all covariant representations of $A$.
(Actually, we choose one Hilbert space of sufficient large cardinality
which can carry all possible representations of $\F(G,A)$ up to equivalence
(meaning that the images of $\F(G,A)$ under two representations are canonically isometrically isomorphic)
and allow only representations on this Hilbert space.)

}
\end{definition}

\begin{definition}
{\rm
The {\em full crossed product} $A \rtimes G$
is the $C^*$-algebra induced by the universal covariant representation of $A$.
}
\end{definition}

Equivalently, $A \rtimes G$ is the norm closure of the image $\phi^\infty(A \rtimes_{\rm alg} G)$ of the universal
representation $\phi^\infty$ of $\F(G,A)$.
Bearing Lemma \ref{lemmaBijectionCovRepF} in mind, by an abuse of language we may also call $\phi^{\infty}$ a covariant representation of $A$.


\begin{lemma} \label{lemmaCovRepInducedMap}
Let $\phi^\infty$ be the universal covariant representation of $A$.
If $\phi$ is another covariant representation of $A$ then
there is a homomorphism $\sigma: A \rtimes G \longrightarrow A \rtimes_\phi G$
such that $\sigma \phi^\infty(x) = \phi(x)$ for all $x \in A \rtimes_{\rm alg} G$.
$$\begin{xy}
\xymatrix{
A \rtimes_{\rm alg} G  \ar[r]^{\phi^\infty} \ar[rd]_{\phi}  &
A \rtimes G  \ar[d]^{\sigma}\\
 &  A \rtimes_{\phi} G
}
\end{xy}$$
\end{lemma}

\begin{proof}
This is clear as $\phi^\infty$ is the direct sum over all representations of $\F(G,A)$, so is larger or equal in norm
in every point $x$ than $\phi$.
\end{proof}


Note that the above full crossed product is for proper semimultiplicicative sets, and so there are differences to
existing crossed products if one considers special categories. Let $(\pi,U,H)$ be a covariant representation
of a $G$-Hilbert $C^*$-algebra $A$. If $G$ is a discrete group, then $U_g U_g^* = U_g^* U_g=U_e$ for all $g \in G$ by Lemma
\ref{corollaryHilbertactionInverses},
but this need not be a unit (we may resolve this difference by requiring $U_e=1$, as optionally done in Sections
\ref{sectionDescent}-\ref{sectionDescent3}).
If $G$ is an inverse semigroup, our crossed product differs from Sieben's crossed product \cite{sieben1997} which
is based on strictly covariant representations in the sense
that $U_g \pi(a) U_g^*= \pi(g(a))$. We are however consistent with Khoshkam--Skandalis' definition \cite{1061.46047},
see Lemma \ref{lemmaInvSemigroupKhoshkam}. The precise difference between the latter two crossed products is clarified in \cite{1061.46047}.
If $G$ is a semigroup, then in the existing definitions a semigroup covariant representation consists of isometries $U_g$
which strictly covariantly intertwine the $G$-action, 
see Stacey \cite{stacey}, Murphy \cite{murphy}, Laca \cite{laca2000} and Larsen \cite{larsen}.
Stacey even allows a family of isometries for representations of different multiplicities.
The crossed product of $\N$ by surjective shift maps on $\{0,1\}^\N$ degenerates to $0$ according to Stacey in \cite[Example 2.1(a)]{stacey}
(this affects any crossed product construction induced by strictly covariantly intertwining isometries)
but there is an obvious non-degnerate covariant representation on $B(\ell^2(\{0,1\}^\N))$ in our sense.
In all constructions of this paragraph the full crossed product is (roughly speaking) the enveloping $C^*$-algebra
of the respective class of equivariant representations. 
%



If $\calg$ is a discrete groupoid then $gh =0$ in the groupoid $C^*$-algebra if $g$ and $h$ are incomposeable ($g,h \in \calg$).
Taking into account such an approach to the crossed product, we consider such a variant also for semimultiplicative sets.


\begin{definition}
{\rm
Let $G$ be a general semimultiplicative set.
A covariant representation $(\pi,u,H)$ is called {\em strong} if $u_{g} u_h=0$ for all incomposable pairs $g,h \in G$.
The {\em full strong crossed product} $A \rtimes_{\rm s} G$
is the $C^*$-algebra induced by the universal strong $G$-covariant representation of $A$.
}
\end{definition}

A similar lemma as Lemma \ref{lemmaCovRepInducedMap} holds also for the strong crossed product and the strong covariant representations.

For a semigroup $S$ there exists a crossed product where the ac


\section{Reduced crossed products}

In this section we shall assume that $G$ is an associative semimultiplicative set with left cancellation.
Let $\rho$ be the injective $G$-action on $G$ given by left multiplication ($\rho_g( h) = gh$ in $G$).
It can be extended to an injective $G^*$-action on $G$ (also denoted by $\rho$) by Lemma \ref{lemmaActionGstar}. 
$\rho$ induces an action $\lambda: G \longrightarrow B(\ell^2(G))$ (Definition \ref{defLeftRegRep}).
This action is an action under which $\ell^2(G)$ becomes a $G$-Hilbert space (i.e. a $G$-Hilbert module over $\C$). We shall regard $\ell^2(G)$
as a $G$-Hilbert module (if nothing else is said).
We may extend this action
to a $G^*$-action, and denote this extension also by $\lambda$
(and it is the same action as the extended $\rho$ would induce).
For arbitrary $g$ in $G^*$ and arbitrary $h$ in $G$ we use the abbreviation
$$e_{g h} := \lambda_g(e_h).$$


%
%
%

\begin{definition}
{\rm
If $G$ has left cancellation then a $G$-action $U$ on a $G$-Hilbert module $\cale$
is said to have {\em transferred left cancellation} if
$U_{g}^* U_g U_h = U_h$ for all $g,h \in G$ for which $g h$ is defined.
}
\end{definition}

The last definition is understood to include $G$-Hilbert $C^*$-algebras
(which are special $G$-Hilbert modules).
By sloppy language we shall also say that a $G$-Hilbert module has transferred left cancellation
(rather than the $G$-action itself).

If $G$ is a semigroupoid then $\lambda$ has transferred left cancellation.
Indeed, assume $gh$ is defined and $x \in G$. Since $G$ is a semigroupoid and $gh$ is defined, $(gh)x$ is defined
if and only if $hx$ is defined. Thus $\lambda_{g}^* \lambda_g \lambda_h(e_x)= \lambda_h(e_x)$.

\begin{lemma}   \label{lemmaTransferredILM}
A $G$-action $U$ has transferred left cancellation if and only if for all $g \in G^*$ and all $h \in G$ one has
$U_{g h} = U_{\rho_g(h)}$ whenever $\rho_g(h)$ is defined
(note that $gh \in G^*$ but $\rho_g(h) \in G$).
\end{lemma}

\begin{proof}
Assume the condition holds true. If $\rho_g(h)$ exists for $g,h \in G$ then $\rho_g^* \rho_g(h) = h$ (Lemma \ref{lemmaActionGstar}).
Consequently $U_h = U_{\rho_{g^* g}(h)} = U_{g^* g h}$ by assumption. Thus $U$ has transferred left cancellation.
Assume that $U$ has tranferred left cancellation and
by induction hypothesis on the length of $g$ that $U_{\rho_g(h)} = U_{g h}$, where $g \in G^*, h \in G$ and $\rho_g(h)$ is defined.
Suppose that $t \in G$ and $\rho_{t^* g}(h)$ is defined. Then $gh=\rho_{t t^* g} (h) = \rho_t(\rho_{t^* g}(h)) = \rho_t(x)$ for $x := \rho_{t^* g}(h)$.
Since $U$ has transferred left cancellation, $U_t^* U_t U_x=U_x$. Hence,
$U_{\rho_{t^* g}(h)} = U_{x} = U_{t^* t x}= U_{t^* g h}$. This proves the inductive step.
On the other hand, if $\rho_{tg}(h)$ is defined, then $U_{\rho_{tg}(h)}= U_{\rho_t (\rho_g(h))} = U_{t (\rho_g(h))}
=U_{t} U_{\rho_g(h)} = U_{t} U_{gh}=U_{tgh}$, proving the inductive step again.
\end{proof}

\begin{definition} \label{definitionLeftRegular}
{\rm
Suppose that $A$ is a $G$-Hilbert $C^*$-algebra,
$G$ is associative with left cancellation, and $A$ has transferred left cancellation.
Let $\sigma:A \longrightarrow B(H)$ be a faithful nondegenerate representation (without $G$-action)
of $A$ on a Hilbert space $H$.
The {\em left reduced crossed product} $A \rtimes_{r} G$ is the $C^*$-algebra
induced by the {\em left regular} covariant representation $(\pi,u, H \otimes \ell^2(G))$ of $A$
given by
\begin{eqnarray*}
\pi(a) (\xi_h \otimes e_h)  &=& \sigma \big (h^*(a) \big) \xi_h  \otimes e_h,\\
u(g) (\xi_h \otimes  e_h)  &=&  \xi_h \otimes \lambda_g(e_{h})
\end{eqnarray*}
for all $a \in A , \xi_h \in H$ and $g, h \in G$.
}
\end{definition}


\begin{lemma}
The left regular representation (Definition \ref{definitionLeftRegular}) is indeed covariant.
\end{lemma}

\begin{proof}
We need to check Defintion \ref{definitionGstarEquivariance} and demonstrate only (\ref{equivariantRep3}).
Let $\hat \alpha$ denote the $G^*$-action on $A$.
By Lemma \ref{lemmaHilbertCstarGstar} (i) and Lemma \ref{lemmaTransferredILM} we have
\begin{eqnarray*}
&& u_g \pi(a) u_g^* (\xi \otimes e_h) \,\,=\,\, u_g \pi(a) (\xi \otimes e_{\rho_{g^*}(h)})  \,\,=\,\, 
u_g\, \big(\sigma\big(\hat \alpha_{\rho_{g^*}(h)}^*(a) \big)\xi \otimes e_{\rho_{g^*}(h)} \big ) \\
&=& u_g\, \big(\sigma\big(\hat \alpha_{g^* h}^*(a) \big)\xi \otimes e_{\rho_{g^*}(h)} \big)
\,\,=\,\, \sigma\big(\hat \alpha_{h^* g}(a) \big)\xi \otimes e_{\rho_{g g^*}(h)}\\
&=& \sigma\big(\hat \alpha_{h^* g g^* g}(a) \big)\xi \otimes e_{\rho_{g g^*}(h)}
\,\,=\,\, \pi(g(a)) u_g u_g^* (\xi \otimes e_h)
\end{eqnarray*}
for all $g \in G^*$ and $h \in G$.
\end{proof}

Obviously, $u$ of Definition \ref{definitionLeftRegular} is the diagonal $G$-action $1 \otimes \lambda$.
We are going to show that the definition of $A \rtimes_r G$ is actually independent of $\sigma$.

We shall recall three lemmas which can all be found in Kasparov \cite{kasparov1981}, pages 522-523.
Only Lemma \ref{lemmakuerzetensor} is somewhat extended (cf. Lance \cite[Proposition 2.1]{lance}).

\begin{lemma} \label{lemmakuerzetensor}
Let $X$ be a Hilbert module, $A$ a $C^*$-algebra and
$\pi: A \longrightarrow \call(X)$ a non-degenerate homomorphism. 
Then there is an isomorphism
$$\rho: A \otimes_{A} X \longrightarrow X : \rho(a \otimes x ) = \pi(a) x.$$
If $T \in \call(A)$ then $T \otimes 1 = \rho^{-1} \hat \pi(T) \rho$,
where $\hat \pi: \call(A) \longrightarrow \call(X)$ denotes the strictly continuous extension of $\pi$.
\end{lemma}

\begin{lemma} \label{lemmaActingOnSpaceTensor}
If $X$ and $H$ are Hilbert modules over $C^*$-algebras $B_1$ and $B_2$, respectively, and $B_1 \rightarrow \call(H)$ is an injective homomorphism
then $\mu: \call(X) \rightarrow \call(X \otimes_{B_1} H)$, $\mu(T) = T \otimes 1$
is an injective homomorphism.
\end{lemma}

\begin{lemma} \label{lemmaTensorSkewwithInterior}
If $E_1, \ldots, E_4$ are Hilbert $B_i$-modules and $B_1 \rightarrow \call(E_3), B_2 \rightarrow \call(E_4)$
are homomorphisms then
$$(E_1 \otimes E_2) \otimes_{B_1 \otimes B_2} (E_3 \otimes E_4) \cong (E_1 \otimes_{B_1} E_3) \otimes  (E_2 \otimes_{B_2} E_4).$$
\end{lemma}

For a $G$-Hilbert $C^*$-algebra $A$ let $A \otimes \ell^2(G)$ denote the skew tensor product of $G$-Hilbert modules.
We make it a $G$-Hilbert module over $A \otimes \C \cong A$ under the diagonal action
$1 \otimes \lambda$.

\begin{lemma}  \label{lemmaLeftRegActionHilbertModule}
Consider the setting of Definition \ref{definitionLeftRegular}.
There is an injective
$*$-homomorphism
$$\zeta:A \rtimes_r G \longrightarrow \call(A \otimes \ell^2(G))$$
induced by the covariant representation $\phi:A \rtimes_{\rm alg} G \longrightarrow \call(A \otimes \ell^2(G))$
given by
\begin{eqnarray*}
\phi(a) (x_h \otimes e_h) &=& h^*(a)x_h \otimes e_h,\\
\phi(g) &=& 1 \otimes \lambda_g,
\end{eqnarray*}
for all $a, x_h \in A$ and $g,h \in G$.
%
%
\end{lemma}

\begin{proof}
Let $\phi_r$ be the representation of $A \rtimes_{\rm alg} G$ induced by the left regular representation
(Definition \ref{definitionLeftRegular}).
Let $\sigma:A \longrightarrow B(H)$ be a faithful and non-degenerate representation (without $G$-action) of $A$ on a Hilbert space $H$.
We aim to show that there is a commutative diagram
$$\begin{xy}
\xymatrix{
A \rtimes_{\rm alg} G \ar[r]^\phi  \ar[dr]_{\phi_r} & \call \big (A  \otimes \ell^2(G) \big ) \ar[rr]^\mu   \ar[d]^\kappa & & \call \big ( (A \otimes \ell^2(G)) \otimes_{A \otimes \C} (H \otimes \C) \big)
\ar[d]^{\mu_1} \\
& \call \big (H \otimes \ell^2(G) \big )  & & \call \big ((A \otimes_A H)  \otimes  (\ell^2(G) \otimes_\C \C) \big )  \ar[ll]^{\mu_2}
}
\end{xy}$$
Here, $\mu$ is the injective homomorphism of Lemma \ref{lemmaActingOnSpaceTensor},
and $\mu_1$ and $\mu_2$ denote the isomorphisms induced by the isomorphisms of 
Lemma \ref{lemmaTensorSkewwithInterior} and Lemma \ref{lemmakuerzetensor}, respectively.
Define $\kappa := \mu_2 \mu_1 \mu$, which is injective.
We are going to analyse $\kappa (\phi (a \rtimes g))$.
We write an element $\xi \in H$ as $\sigma(a_0) \xi_0$ for $a_0 \in A$ and $\xi_0 \in H$ by Lemma \ref{lemmakuerzetensor}.
We shall write down, step by step, how $\phi(a \rtimes g)$ transforms under $\kappa$.
Let $g \in G^*, h \in G, a \in A_g,x_h \in A$ and $\xi \in H$.
We have
\begin{eqnarray*}
\phi \big(a \rtimes g \big) \big (x_h  \otimes e_h \big )  &=&  (gh)^*(a) x_h \otimes e_{g h}\\
\mu \phi \big(a \rtimes g  \big) \big (  (x_h  \otimes e_h)  \otimes (\xi \otimes 1_\C) \big)  &=&  ((gh)^*(a) x_h \otimes e_{g h}) \otimes (\xi \otimes 1_\C) \\
\kappa \phi \big(a \rtimes g \big) \big (  \sigma(x_h) \xi  \otimes e_h \big)  &=&  \sigma \big ( (gh)^*(a) \big)  \sigma( x_h) \xi  \otimes e_{g h}\\
\kappa \phi \big(a \rtimes g \big) \big (  \overline \xi  \otimes e_h \big)  &=&  \sigma \big ( (gh)^*(a) \big)  \overline \xi  \otimes e_{g h}\\
&=&  \phi_r \big(a \rtimes g \big) \big (  \overline \xi  \otimes e_h \big)
\end{eqnarray*}
In the last step we have set $\overline \xi:=\sigma(x_h) \xi$ (Lemma \ref{lemmakuerzetensor}).
We have checked that $\phi_r = \kappa \phi$. This shows that $\overline{\phi(A \rtimes_{\rm alg} G)}$
is isomorphic to $A \rtimes_r G$, and we set $\zeta := \kappa^{-1}$.
\end{proof}

\begin{corollary}  \label{corollaryLeftRegReprIndepenceRepPi}
The definition of the left reduced crossed product in Definition \ref{definitionLeftRegular} does not depend on $\sigma$.
\end{corollary}

For the rest of this section we consider the following assumptions.
Let $L: \F(G,A) \longrightarrow B(H \otimes \ell^2(G))$ be the left regular representation.
Then $L(G^*)$ is an inverse semigroup.
Suppose that the $G^*$-action on $A$ factors through $L(G^*)$ via an inverse semigroup homomorphism $\mu$.
$$\begin{xy}
\xymatrix{
G^*  \ar[r]^{L}  \ar[rd]_{\hat \alpha}  & L(G^*)  \ar[d]^\mu\\
 &  \Endo (A)
}
\end{xy}$$
(For instance, when the $G$-action on $A$ is trivial.)
Then $\mu$ defines a $L(G^*)$-action on $A$.
Suppose further that $L$ is injective on $A$.


\begin{lemma}
There is an isomorphism
\begin{eqnarray}   \label{mapGamma}
\gamma: L \big( \F(G,A) \big) \longrightarrow
\F \big (L(G^*), A \big): \quad \gamma(L(a)) = a, \quad \gamma( L(g))=L(g),
\end{eqnarray}
where $a \in A$ and $g \in G^*$, which restricts to an isomorphism
\begin{eqnarray}   \label{mapGammaRestrict}
L( A \rtimes_{\rm alg} G) \longrightarrow  A \rtimes_{\rm alg} L(G^*).
\end{eqnarray}
\end{lemma}

\begin{proof}
Note that in $\F(L(G^*),A)$ we have $L(g) a = \mu_{L(g)} (a) L(g) = \hat \alpha_g(a) L(g)= g(a) L(g)$.
At first we shall show that $\gamma\circ L$ is a representation of $\F(G,A)$. To this end we need to check that
the relations (\ref{ident0})-(\ref{ident2}) are respected by $\gamma\circ L$.
We only show (\ref{ident2}),
$$\gamma L(g) \gamma L(a) (\gamma L(g))^* = L(g) a L(g)^* = g(a) L(g) L(g)^* = \gamma L(g(a)) \gamma L(g) (\gamma L(g))^*.$$
Since $L$ and $\gamma \circ L$ are homomorphisms, $\gamma$ is a homomorphism.
 
We need to show that there is an inverse map $\sigma$ for $\gamma$, where $\sigma(a) = L(a)$ and $\sigma(L(g))=L(g)$.
Again we have to check that the relations (\ref{ident0})-(\ref{ident2}) are respected by $\sigma$. For instance,
$$\sigma(L(g)) (\sigma(L(g)))^* \sigma(L(g)) = L(g) L(g)^* L(g) = L(g) = \sigma(L(g)),$$
since $L(g)$ is a partial isometry.
\end{proof}

\begin{corollary}  \label{corollaryIsomorphRedInvSemiCrossed}
If the given $C^*$-norm on $L(A \rtimes_{\rm alg} G)$ is the maximal (covariant) one, then 
\begin{eqnarray}  \label{equivalRedInvSemi}
A \rtimes_r G \cong A \rtimes L(G^*).
\end{eqnarray}
\end{corollary}

\begin{proof}
Let $\gamma_0$ be the isomorphism (\ref{mapGammaRestrict}) and endow domain and range with the norms from $A \rtimes_r G$ and $A \rtimes L(G^*)$, respectively.
Since $\gamma_0^{-1}$ is the restriction of $\gamma^{-1}$, (\ref{mapGamma}), by Lemma \ref{lemmaBijectionCovRepF}
it is a covariant representation of $A \rtimes_{\rm alg} L(G^*)$. Thus $\gamma_0^{-1}$ is norm-decreasing.
On the other hand, $\gamma_0$ is a (covariant) representation of $L(A \rtimes_{\rm alg} G)$, which by assumption must decrease in norm.
Thus $\gamma_0$ is an isometry and extends continuously to (\ref{equivalRedInvSemi}).
%
%
\end{proof}


The last corollary may be useful to translate reduced crossed products to inverse semigroup crossed products, for which there exist more Baum--Connes theory (see for instance \cite{burgiKKrDiscrete} and \cite{burgiGreenJulg}).
For example,
some Toeplitz graph $C^*$-algebras for graphs $\Lambda$ are reduced $C^*$-algebras $\C \rtimes_r \Lambda^*$ (via the so so-called path space representation, see for instance \cite{simsThesis}).
By a Cuntz--Krieger uniqueness theorem (the $C^*$-norm on $L(\C \rtimes_{\rm alg} \Lambda^*)$ is unique), Corollary \ref{corollaryIsomorphRedInvSemiCrossed} applies immediately.





\section{Representations of $\ell^1(G)$}   \label{sectionReprL1}

Write $\ell^1(G,A)$ for the completion of $C_c(G^*,A)$
under the norm
$\|\sum_{g \in G^*} a_g g\|_{1} = \sum_{g \in G^*} \|a_g\|$.
For $a , b \in C_c(G^*,A)$ the estimate
$\|a b\|_1 \le \|a\|_1 \|b\|_1$ is easy.

\begin{lemma}
$\ell^1(G,A)$ is a Banach $*$-algebra.
\end{lemma}


%

A {\em representation} of $\ell^1(G,A)$ is a norm bounded $*$-homomorphism $\pi: \ell^1(G,A) \longrightarrow B(H)$, where $H$ is a Hilbert space.

\begin{proposition}
%
If $\ell^1(G,A)$ has an approximate unit then
a representation of $\ell^1(G,A)$ is realized by a covariant representation of $A$, and vice versa.
(It need not be a bijection, see \cite[Remark, p.271]{1061.46047}.)



Consequently, if $\ell^1(G,A)$ has an approximate unit then a representation of $A \rtimes_{\rm alg} G$ extends to $\F(G,A)$ if and only if it is covariant if and only if it is bounded
in $\ell^1(G,A)$-norm.
\end{proposition}

\begin{proof}
We essentially follow Pedersen's book \cite{pedersen}, Proposition 7.6.4.
Let $\pi: \ell^1(G,A) \rightarrow B(H)$ be a representation on a Hilbert space $H$.
It is a direct sum of a non-degenerate representation and the null-representation.
We may ignore the null-part, which we can then add to the covariant representation of $A$ again, and vice versa, and assume
that $\pi$ is non-degenerate.
The left and right multiplications of elements $z \in A \rtimes_{\rm alg} G$ by elements $a \in A , g \in G$ in the algebra $\F(G,A)$,
that is, $z \mapsto a z$ would be the operator given by left multiplication by $a$,
induce bounded linear maps (even centralizers) $L_a, L_g, R_a, R_g$ from $\ell^1(G,A)$ into itself.
Let $(y_i) \subseteq \ell^1(G,A)$ be a given approximate unit.
Since $\pi$ is non-degenerate, $\pi(\ell^1(G,A))H$ is dense in $H$.
Since for each $\eta =\pi(x) \xi$ ($x \in \ell^1(G,A), \xi \in H$) one has $\|\eta - \pi(y_i) \eta \| \le \|\pi(x - y_i x) \xi \|
\le \|x - y_i x\|_1 \|\xi\|
\rightarrow 0$ for $i \rightarrow \infty$, $\pi(y_i)$ converges strongly to the unit of $B(H)$.
Similarly, for all $a \in A$ and $x \in \ell^1(G,A)$
the Cauchy criterium
$\| \pi(a y_i - a y_j) \pi(x) \xi\| \le \varepsilon$
for all $i,j \ge i_0$ shows that $\pi(a y_i)=\pi(L_a(y_i))$ has a strong limit point $\sigma(a)$.
Hence $\pi(a x)= \lim_i \pi(a y_i x) = \lim_i \pi(ay_i) \pi(x)=\sigma(a) \pi(x)$.
Since $\|\pi(y_i a - a y_i) \pi(x) \xi\| \rightarrow 0$ for $i \rightarrow \infty$,
$\sigma(a) =\lim_i \pi(L_a(y_i)) =\lim_i \pi(R_a(y_i))$ (strong limits).
In the same manner we define $U_g = \lim_i \pi(L_g(y_i)) = \lim_i \pi(R_g(y_i))$ (strong limits),
and one has $\pi(g x)= U_g \pi(x)$  for $g \in G$.
Analogously we define $U_g^*$ for $g \in G$.
A direct check shows that $(\sigma,U,H)$ is a $G$-covariant representation of $A$.
For instance,
$$U_g \sigma(a) U_g^* \pi(x) =  U_g \sigma(a) \pi(g^* x) = \pi(g a g^* x)
= \pi(g (a) g g^* x) = \sigma(g(a))U_g U_g^* \pi(x),$$
and replacing $x$ by $y_i$ and taking the limit yields (\ref{equivariantRep3}).
In particular we have $\pi(a_g g) = \sigma(a_g) U_g$, which extends by norm continuity to $\ell^1(G,A)$.
This shows that $\pi$ will be assigned to $(\sigma,U,H)$.
On the other hand, starting with a representation $(\sigma,U,H)$
we define a representation $\pi$ of $\ell^1(G,A)$ by $\pi(a_g g) = \sigma(a_g) U_g$.
\end{proof}

\begin{corollary}   \label{corollaryEll1}
If $\ell^1(G,A)$ has an approximate unit then $A \rtimes G$ (resp.
$A \rtimes_s G$) is the $C^*$-algebra generated by the universal
(resp. universal strong) representation of $\ell^1(G,A)$.
\end{corollary}

\begin{lemma}  \label{lemmaInvSemigroupKhoshkam}
If $G$ is an inverse semigroup then $A \rtimes G$ coincides with Khoshkam and Skandalis' definition in \cite{1061.46047}, so is
the envelopping $C^*$-algebra of $\ell^1(G,A)$.
\end{lemma}

\begin{proof}
Let $\alpha$ be any bounded representation of $\ell^1(G,A)$ on Hilbert space. Then it factors through Khoshkam--Skandalis' crossed product $A \rtimes G$.
Any $C^*$-representation of $A \rtimes G$ is realized as a covariant representation of $A$ by \cite[Theorem 5.7.(b)]{1061.46047}, so the same must
be true for $\alpha$.  

%

Hence, a $C^*$-representation of $\ell^1(G,A)$ is $G$-covariant.
But then, since every $G$-covariant $C^*$-representation of $A \rtimes_{\rm alg} G$ is obviously bounded in $\ell^1(G,A)$-norm,
$A \rtimes_{\rm alg} G$ and $\ell^1(G,A)$ have the same universal $G$-covariant representation (which induces the $C^*$-crossed products).
\end{proof}

\section{$KK^G$ for unital $G$}

\label{sectionCompareWithGroupsKK}

In this section we will compare Kasparov's equivariant $KK$-theory with semimultiplicative sets equivariant $KK$-theory
when $G$ happens to be a group.
We shall then also introduce a unital version of $KK^G$-theory for unital semimultiplicative sets $G$,
where we let the unit of $G$ act as the identity on Hilbert modules and $C^*$-algebras.

Recall that two cycles $(\cale,T)$ and $(\cale,T')$ in $\E^G(A,B)$ are {\em compact perturbations} of each
other if $a (T - T') \in \calk(\cale)$ for all $a \in A$, and that then the straight line segment
from $T$ to $T'$ is an operator homotopy; in particular
$(\cale,T)$ and $(\cale,T')$ are homotopic in the sense of $KK^G$-theory
(see \cite{burgiSemimultiKK}).
We will denote Kasparov's equivariant $KK$-theory for groups $G$ (\cite{kasparov1981,kasparov1988}) by $\widetilde{KK^G}(A,B)$.

\begin{proposition} \label{propositionGgroupKK}
Let $G$ be a group (or a unital semimultiplicative set, see Remark \ref{remarkUnitalKK}). Let $A$ and $B$ be Hilbert $C^*$-algebras where the unit of $G$ acts identically on $A$ and $B$, respectively. Then
\begin{eqnarray*}
{KK^G} (A,B) &\cong& \widetilde{KK^G}(A,B) \oplus \widetilde{KK}(A,B).
\end{eqnarray*}
\end{proposition}

\begin{proof}
The proof of this proposition (which had also been suspected by the author) was indicated by an unkonwn referee.
Let $(\cale,T)$ be a cycle in $\E^G(A,B)$.
By Lemma \ref{IdemLemma} and Corollary \ref{corollaryHilbertactionInverses},
$U_e$ is a projection and a unit for all $U_g$, and $U_{g^{-1}}= U_g^*$, and so $U_g U_{g}^*= U_g^* U_g =U_e$ for all $g \in G$.
Hence, $KK^G(A,B)$ and $\widetilde{KK^G}(A,B)$ differ only by the fact that $\widetilde{KK^G}(A,B)$
is build up by cycles $(\cale,T) \in \widetilde{\E^G}(A,B)$
where $U_e$ acts identically on $\cale$.

Denote $u=U_e$. We aim to show that the map
\begin{eqnarray*}
&&\Phi_{A,B}:\E^G(A,B) \longrightarrow \widetilde{\E^G}(A,B) \oplus \widetilde{\E}(A,B)\\
&&\Phi_{A,B} (\cale,T) = (u \cale, uTu) \oplus ((1-u) \cale, (1-u)T(1-u))
\end{eqnarray*}
induces an isomorphism in $KK$-theory.
Homotopic elements in $\E^G(A,B)$ become homotopic elements in the image of $\Phi_{A,B}$ via the map $\Phi_{A,B[0,1]}$
(because $U_e \otimes \alpha_e = U_e \otimes 1$ on $\cale \otimes_{B[0,1]} B$).
The map $\Phi_{A,B}$ has an obvious canoncial inverse map $\Phi_{A,B}^{-1}$, which also respects homotopy.
Obviously we have $\Phi_{A,B} \Phi_{A,B}^{-1} =1$.
On the other hand,
$$\Phi_{A,B}^{-1} \Phi_{A,B}(\cale,T)= (\cale, uTu+(1-u)T(1-u))$$
is just a compact perturbation
of $(\cale,T)$. Hence also $\Phi_{A,B}^{-1} \Phi_{A,B} \sim 1$.
\end{proof}

\begin{remark} \label{remarkUnitalKK}
{\rm The above revealed difference between Kasparov's theory and
ours seems natural as usually lacking an identity in $G$,
$G$-actions are allowed to act degenerate on $C^*$-algebras or
Hilbert modules. This is reflected in the $KK^G$-theory. If,
however, one considers unital $G$'s one can neutralize the
difference to Kasparov's theory by assuming that the unit $1_G$ of
$G$ always acts as the identity on Hilbert modules and Hilbert $C^*$-algebras. Then
the whole $KK^G$-theory of \cite{burgiSemimultiKK} goes through
under this modification (so one has also an associative Kasparov product). This is clear as we only have to
take care that all used constructions of $G$-Hilbert modules
respect the unitization, and these are the tensor products
and the direct sum where it is obvious. Furthermore, one has to
ensure that under modified $KK^G$-theory the class $1$ in
$KK^G(\C,\C)$ associated to the cycle $(\C,0)$ (as used in Section
7 of \cite{burgiSemimultiKK}) exists; but this is also clear.
%
%
Actually, the proof of Proposition \ref{propositionGgroupKK} works (without essential modification)
for any unital semimultiplicative set $G$, that is, $KK^G$ is the
direct sum of the unital version of $KK^G$, where the unit of $G$
acts fully on Hilbert $C^*$-algebras and Hilbert bimodules, and
Kasparov's $\widetilde{KK}$.
}
\end{remark}


\section{Inversely generated semigroups}   \label{sectionInvGenSemigroups}


\begin{definition}   \label{definitionPartialIsometry}
{\rm
We call an element $g$ of an involutive semigroup $\overline G$ a {\em partial isometry} if it is invertible
with respect to the involution, that is, if $g g^* g = g$.
}
\end{definition}

Note that if $s$ is a partial isometry then $s^*$ is also one. Consequently, the set of partial isometries
of an involutive semigroup is self-adjoint.


\begin{definition}
{\rm
An {\em inversely generated semigroup} is an involutive semigroup $\overline G$
which is generated by its partial isometries.
In other words, for every $g \in \overline G$ there exist partial isometries $s_1,\ldots,s_n \in \overline G$
such that
$g = s_1 \ldots s_n$.
}
\end{definition}

The standard example for an inversely generated semigroup is the involutive semigroup $G^*$ for a semimultiplicative set $G$ (Definition \ref{simple_equivalences}).
(The set of partial isometries of $G^*$ might differ from $G$, since there could exist more partial isometries.)

\begin{definition}
{\rm
A {\em $*$-morphism} between involutive semigroups
is a map respecting the multiplication and the involution.
A {\em $*$-antimorphism} between involutive semigroups is an involution respecting
semigroup antimorphism.
}
\end{definition}

We shall write $G$ for the set of partial isometries of an inversely generated semigroup $\overline G$.
$G$ is a semimultiplicative set which usually is not associative.
(One can easily construct examples where $st \in G$ and $(st)u \in G$ are partial isometries, but $t u \notin G$ is not one; this contradicts the associativity condition.)




 

\begin{definition}   \label{DefInvGenSemiAlg}
{\rm
A $\overline G$-Hilbert $C^*$-algebra is a semimultiplicative set $G$-Hilbert $C^*$-algebra $A$
where the action maps $\alpha,\alpha^*:G \longrightarrow \Endo(A)$
extend
to a map $\overline \alpha: \overline G \longrightarrow \Endo(A)$
%
\begin{eqnarray}
\overline \alpha(g) &=& \alpha(g), \label{Haction1}\\
\overline \alpha(g^*) &=& \alpha^*(g),  \label{Haction2}\\
\overline \alpha(h k) &=& \overline \alpha(h) \overline \alpha(k)  \label{Haction3}
\end{eqnarray}
for all $g \in G$ and $h,k \in \overline G$. }
\end{definition}

Since $\overline \alpha$ maps into the partial isometries of $A$ which have commuting source and range projections (in the center of the multiplier algebra),
$\overline \alpha$ is actually a $*$-morphism.

\begin{definition}    \label{DefInvGenSemiModule}
{\rm
A $\overline G$-Hilbert module is a Hilbert module which is endowed with a general semimultiplicative
set $G$-action $\alpha$ that extends to a map $\overline \alpha$ via the formulas (\ref{Haction1})-(\ref{Haction3}).
}
\end{definition}

Note that the $G$-action $\overline \alpha$ on a Hilbert module is usually not realized by partial isometries; only the partial isometries
of $\overline G$, that is the elements of $G$, go over to partial isometries
(because a semimultiplicative set $G$-action is always realized by partial isometries).
These partial isometries determine how we have to define the other elements of $\overline G$,
as they can be written as products of elements of $G$. These
products, however, need not be partial isometries on the Hilbert module.

We may equivalently reformulate Definition \ref{DefInvGenSemiAlg} (and similarly Definition \ref{DefInvGenSemiModule}) by saying that the $G^*$-action $\hat \alpha$ on $A$ factors through $\overline G$.
$$\begin{xy}
\xymatrix{
G^*  \ar[r]^{\hat \alpha} \ar[d]^p  &
A \\
\overline G \ar[ur]_{\overline \alpha}   &
}
\end{xy}$$
Here, $p$ is the quotient $*$-morphism determined by $p(g)= g$ for all $g \in G$.
Indeed, if $\alpha$ allows an extension $\overline \alpha$ given by (\ref{Haction1})-(\ref{Haction3}) then
the above diagram commutes. On the other hand, if the above diagram exists, $\overline \alpha$ is an extension of $\alpha$
satisfying (\ref{Haction1})-(\ref{Haction3}).

Because of this fact we view a $\overline G$-Hilbert module also as a $G$-Hilbert module with the property that the induced $G^*$-map factors
through $\overline G$. We say sloppy that the $G$-Hilbert module factors through $\overline G$.

\begin{lemma}  \label{lemmaCycleCondGstar}
Identities (\ref{defCycleCondition}) hold also for all $g \in G^*$.
\end{lemma}

\begin{proof}
We leave the inductive proof to the reader, and sketch only one identity modulo $I_A(\cale)$;
note that $g(\calk(\cale)), g^*(\calk(\cale)) \subseteq \calk(\cale)$ for all $g \in G$. For $g \in G$ and some $h \in G^*$ (given by inductive hypothesis) we have
$$U_g U_h T U_h^* U_g^* \equiv U_g T U_h U_h^* U_g^* \equiv U_g T U_g^* U_g U_h U_h^* U_g^* \equiv T U_g U_h U_h^* U_g^*.$$
\end{proof}

A $G$-equivariant homomorphism $\pi:A \longrightarrow \call(\cale)$ (Definition \ref{definitionGstarEquivariance}) is automatically $G^*$-equivariant by Lemma
\ref{lemmaHilbertCstarGstar} (ii). Thus it is also $\overline G$-equivariant when the appearing $G$-Hilbert module $\cale$ and $G$-Hilbert $C^*$-algebra $A$
factor through $\overline G$.
Such a similar fact can also be said for a cycle $(\cale,T) \in \E^G(A,B)$.
By Lemma \ref{lemmaCycleCondGstar}, identites (\ref{defCycleCondition}) hold also for $g \in \overline G$ if
all Hilbert modules $\cale, A$ and $B$ factor through $\overline G$.
The following definition seems thus natural.


\begin{definition}   \label{defKKInvGenSemi}
{\rm We define $\overline G$-equivariant $KK$-theory in the same way as $KK^G$-theory but with the addition that all appearing $G$-Hilbert modules
and $G$-Hilbert $C^*$-algebras factor through $\overline G$.
}
\end{definition}

In other words, $KK^{\overline G}$-theory is build up by $\overline G$-Hilbert modules rather than by $G$-Hilbert modules as in $KK^G$-theory.

It is easy to see that the category of $\overline G$-Hilbert modules is stable under tensor products and direct sums.
Also, any Hilbert module is a $\overline G$-Hilbert module under the trivial $\overline G$-action.
We have thus checked
that all discussion and theorems
like the Kasparov product in \cite{burgiSemimultiKK} carry over from $KK^G$ to $KK^{\overline G}$
(compare with Remark \ref{remarkUnitalKK}).



We say a representation $\phi: \F(G,A) \longrightarrow B(H)$ factors through $\overline G$ if the restriction map $\phi|_{G^*}$ factors through $\overline G$. (Analogously and equivalently, the $G$-equivariant representation $(\phi|_A,\phi|_G,H)$ is said to factor through $H$).
We prefer it to view a crossed product of $A$ by $\overline G$
as a special crossed product of $A$ by $G$
and introduce the following definition.

%
%
\begin{definition}   \label{defCrossedProductInvGenSemi}
{\rm
The full crossed product $A \rtimes \overline G$ is the norm closure of $\phi^{\overline G}(A \rtimes_{\rm alg} G)$, where
$\phi^{\overline G}$ denotes the universal representation of $\F(G,A)$ which factors through $\overline G$.
}
\end{definition}


\section{Hilbert bimodules over full crossed products}

\label{sectionDescent}

In the remainder of this paper we are going to prove the descent homomorphism.
In this and the remaining sections $H$ and $G$ denote discrete countable semimultiplicative sets.
We may either assume that
$H$ and $G$ have units $1_H$ and $1_G$ and treat everything in the unital world of $KK$-theory
(see Remark \ref{remarkUnitalKK}),
and define the product of $H$ and $G$ by $H \times G$;
or we consider the non-unital version, in this case defining the product of $H$ and $G$ as the
semimultiplicative set $H \sqcup G \sqcup H \times G$ with multiplications
$$h \cdot g := (h,g), \,  h \cdot (h', g') := (h h' , g'),\, (g, h) \cdot(g', h') := (g g', h h')$$
and so on
for $h,h'\in H$ and $g,g' \in G$,
and denote this product, by sloppy but suggestive notation,
still as $H \times G$.
In any case, a morphism $H \times G \longrightarrow K$ is determined by its restriction to $H$ and $G$,
where $H$ and $G$ are identified with $H \times 1_G$ and $1_H \times G$, respectively, in the unital case.

For all $H \times G$-actions on Hilbert modules or $C^*$-algebras we
require that the induced $H^*$-actions and $G^*$-actions (in the
sense of Lemmas \ref{lemmaHilbertmoduleGstar} and
\ref{lemmaHilbertCstarGstar}) commute: the point is that $h^*$
may not commute with $g$ otherwise $(h \in H, g \in G$). This
requirement also affects the definition of $KK^{H \times G}$, and in
this sense the notion $KK^{H \times G}$ is suggestive but sloppy.
(See the discussion in Remark \ref{remarkUnitalKK} why we can
slightly adjust equivariant $KK$-theory: Actually we only need
stability under tensor products, direct sums, and the
existence of $1=(\C,0)$ in $KK^G(\C,\C)$.)

Let $l \in \{\emptyset,s,r,i\}$ and $D$ a $G$-Hilbert $C^*$-algebra.
Let $\phi_{D,G,l}$ be the representation of $\F(G,D)$ induced by the
universal $G$-covariant representation (in case that $l=\emptyset$), or the universal
strong $G$-covariant representation (when $l=s$), or the reduced representation of $D$
(when $l=r$).

The case $l=i$ requires that we are given an inversely generated semigroup denoted by $\overline{G}$ and $\overline H$, and
$G$ and $H$, respectively, denote their subsets of partial isometries.
In this case all appearing $G$-Hilbert modules and $G$-Hilbert $C^*$-algebras are supposed to factor
through $\overline G$ (and similarly so for $H$ and $G \times H$) in accordance to Definition \ref{defKKInvGenSemi}.
If $l=i$ then we need to work with $\overline{G}$-equivariant $KK$-theory, that is, $KK^{G \rtimes H}$ means then actually 
$KK^{{\overline G} \rtimes {\overline H}}$ in this and subsequent sections.
Moreover, $\phi_{D,G,i}$ denotes the universal $\overline{G}$-factorizing $G$-covariant representation
of $D$, and $D \rtimes_i G$ will stand for $D \rtimes \overline{G}$
(Definition \ref{defCrossedProductInvGenSemi}).

We shall sometimes write $\phi_l$ rather than
$\phi_{D,G,l}$ if $D$ and $G$ are clear from the context.
Recall that
$$D \rtimes_l G \cong \overline{\phi_{D,G,l}(D \rtimes_{\rm alg} G)}.$$
We denote
$$G'=\{g , g^* \in G^*|\, g \in G\}.$$
%

If $l=r$ then we deal with the reduced crossed product, and in this case we assume that
$G$ is an associative semimultiplicative set with left cancellation, and all $G$-Hilbert modules
and $G$-Hilbert $C^*$-algebras have transferred left cancellation.
So in this sense we also have a modified $KK^G$-theory as we adapt it in the sense that it is build up by modules with left transferred cancellation (confer Remark \ref{remarkUnitalKK} why we can easily slightly adapt $KK$-theory).
However, we do not require cancellation for $H$ or its actions.
If $l=r$ then we assume that $B=\C$ equipped with the trivial $G$-action.

We will assume that $G$ has a unit, partially because of
non-degenerateness concerns as in Lemma \ref{lemmaIsomorphE12}.
Nevertheless we shall sometimes try to avoid using a unit.

Assume that $A,B$ are $(H \times G)$-Hilbert $C^*$-algebras
and $\cale$ is a $(H \times G)$-Hilbert $B$-module.
The $G$-action on $\cale$ is denoted by $U$.

\begin{lemma} \label{lemmaHHilbertCstar}
(i) $B \rtimes_l G$ is a $H \times G$-Hilbert $C^*$-algebra (where the $G$-action is trivial).

(ii)
Under a different $H \times G$-action denoted by $V$, $B \rtimes_l G$ is a $H \times G$-Hilbert
module over the $H \times G$-Hilbert $C^*$-algebra $B \rtimes_l G$.
This Hilbert module is denoted by $B \rtimes_l^{\rm Mod} G$.
\end{lemma}

\begin{proof}
(i) Let $\phi_l = \phi_{B, G,l}$.
We endow $B \rtimes_l G$ with the $H \times G$-Hilbert $C^*$-action
\begin{equation}   \label{actionalphahg}
\alpha_{h \times g} \big( \phi_l(b_k k) \big) = \phi_l \big(h(b_k) k \big) =: \psi(b_k  k)
\end{equation}
for $k \in G^{*}, b_k \in B_k$ and $h \times g \in (H \times G)'$.
(So the $G$-action is trivial.)
We claim that $\psi:\F(G,B) \longrightarrow B \rtimes_l G$ is a representation. We need to show that $(\psi|_B, \psi|_G)$ is $G$-covariant,
where $\psi(b)= \phi_l(h(b))$ and $\psi(g)= \phi_l(g)$.
Let us check (\ref{equivariantRep1}). In $\phi_l(\F(G,B))$ we have
\begin{eqnarray*}
&& \psi(g) \psi(g)^*  \psi(b) = \phi_l(g) \phi_l(g)^* \phi_l(h(b)) = \phi_l(g g^* h(b) ) = \phi_l(g g^{*}( h(b)) g g^*)\\
&=& \phi_l(h(b) g g^*) = \psi(b) \psi(g) \psi(g)^*,
\end{eqnarray*}
where
$g g^{*}(b) g g^* = b g g^*$ is identity (\ref{ident2}) (Lemma \ref{lemmaIsomFGA} (ii)).

In case that $l$ indicates the full or full strong crossed product,
the map $\alpha_{h \times g}$ extends to a well defined endomorphism of $B \rtimes_l G$ by Lemma
\ref{lemmaCovRepInducedMap}.
For the reduced crossed product we see the boundedness of $\alpha_{h \times g}$ by direct evaluation of the left regular representation of
Definition \ref{definitionLeftRegular}:
one computes
$$\Big\|\phi_r \Big(\sum_{k \in G^*} h(b_k ) k \Big )\xi \Big\| \le \Big \| \phi_r \Big(\sum_{k \in G^*} b_k k \Big ) \xi \Big\|$$
for all
$\xi \in H \otimes \ell^2(G)$.

It remains to check the identities of Definition \ref{definitionGHilbertModule}
to see that $\alpha$ is a $G \times H$-action on $B \rtimes_l G$. For instance, by Lemma \ref{lemmaHilbertCstarGstar} (iii) one has
\begin{eqnarray*}
&&  \big\langle \alpha_{h \times g} \, \phi_l(b_k k) , \phi_l(c_m m)  \big\rangle \,\,=\,\, \phi_l \big(k^* \,h(b_k^*) c_m \, m\big)
\,\,=\,\, \phi_l \big(k^* \,h(b_k^*\, h^* (c_m)) \, m\big)\\
&=& \alpha_{h \times g}\,\big \langle \phi_l(b_k k) , \alpha_{h \times g}^* \, \phi_l(c_m m) \big) \big \rangle.
\end{eqnarray*}

(ii)
We make $B \rtimes_l G$ a Hilbert $B \rtimes_l G$-module $B \rtimes_l^{\rm Mod} G$ with inner product $\langle x,y\rangle = x^* y$
and
$(H \times G)$-Hilbert $B \rtimes_l G$-module action
\begin{equation}    \label{descentVaction}
V_{h \times g} \big( \phi_l(b_k k) \big) = \phi_l \big(g \big(h(b_k) \big) g k \big )
\end{equation}
for all $k \in G^{*},b_k \in B_k$ and $h \times g \in (H \times G)'$.
Note that
\begin{equation}    \label{descentVaction2}
V_{h \times g} \big(\phi_l(x) \big)= \phi_l(g) \, \alpha_h (\phi_l(x) )
\end{equation}
($x \in A \rtimes_{\rm alg} G$),
which shows the boundedness of $V_{h \times g}$.
Then $V$ is an action, and we shall demonstrate only one rule:
$$\big \langle V_g \phi_l(x), \phi_l(y) \big  \rangle = \phi_l(x^*) \phi_l(g^*) \phi_l(y) =
 \big \langle \phi_l(x), V_g^* \phi_l(y) \big \rangle  = \alpha_g \big \langle \phi_l(x), V_g^* \phi_l(y) \big \rangle.$$
\end{proof}

\begin{lemma}   \label{lemmaLeftMultiCross}
There is a $H \times G$-equivariant homomorphism $\tau:B \longrightarrow \call(B \rtimes_l^{\rm Mod} G)$
given by left multiplication, i.e.
$$\tau(b) \big(\phi_l(x)\big ) = \phi_l(b) \phi_l(x)$$
for $b \in B$ and $x \in B \rtimes_{\rm alg} G$.
\end{lemma}

\begin{proof}
We only check (\ref{equivariantRep3})-(\ref{equivariantRep4}).
Let $k \in G^*, g \times h \in (G \times H)', b \in B$ and $c_k \in B_k$.
Then we have
\begin{eqnarray*}
&& V_{g \times h} \, \tau(b) \, V_{g \times h}^* \,\, \phi_l(c_k k)
\,\, = \,\,  V_{g \times h} \, \tau(b) \,\, \phi_l(g^*) \,\phi_l \big(h^{*}(c_k )k \big)\\
&=&  \phi_l(g) \, \phi_l \big( h(b  g^{*} h^{*}(c_k )) g^*k  \big )
\,\,= \,\,  \phi_l \big( g h(b  g^{*} h^{*}(c_k )) g g^*k \big )\\
& =&  \tau  (gh(b) ) \, V_{h \times g} \, V_{h \times g}^* \,\, \phi_l(c_k k).
\end{eqnarray*}
Notice that here we used the requirement that the $G$- and $H$-actions (and their adjoint actions) commute.
\end{proof}

\begin{definition}
{\rm
Define a $H \times G$-Hilbert module over $B \rtimes_l G$ by
$$\cale \rtimes_l G = \cale \otimes_B (B \rtimes_l^{\rm Mod} G)$$
(internal tensor product of $H \times G$-Hilbert modules),
where $B$ acts on $B \rtimes_l^{\rm Mod} G$ by left multiplication (Lemma \ref{lemmaLeftMultiCross}).
}
\end{definition}

By definition, $\cale \rtimes_l G$ is a $H \times G$-Hilbert module
over the $H \times G$-Hilbert $C^*$-algebra $B \rtimes_l G$
under the diagonal action $U \otimes V$ (see \cite[Lemma 4]{burgiSemimultiKK}).
Here, $V$ denotes the $H \times G$-action on $B \rtimes_l G$,
see (\ref{descentVaction}).
Note that if $l= i$, then both $B \rtimes_i G$ and $B \rtimes_i^{\rm Mod} G$ factor through
$\overline H \times \overline G$ under their actions $\alpha$ and $V$ ((\ref{actionalphahg}) and (\ref{descentVaction2})), respectively.
Consequently the tensor product $\cale \rtimes_i G$ factors through $\overline H \times \overline G$.

\begin{proposition} \label{lemmaTheta}
If $l$ indicates one of the full crossed products, i.e. $l \in \{\emptyset, s,i\}$, then
$\cale \rtimes_l G$ is a $H$-Hilbert $(A \rtimes_l G,B \rtimes_l G)$-bimodule.
\end{proposition}

\begin{proof}
$A \rtimes_l G$ is a $H$-Hilbert $C^*$-algebra by Lemma \ref{lemmaHHilbertCstar}.
Let $U \otimes V$ be the diagonal $H \times G$-action
on $\cale \otimes_B (B \rtimes_l^{\rm Mod} G)$.
Note that $U_g \otimes V_g$ is an adjoint-able operator as the $G$-action on
$B \rtimes_l G$ is trivial (see (\ref{actionalphahg})).
Let $\phi_l = \phi_{A,G,l}$.
We define a $*$-homomorphism
$\Theta_l: A \rtimes_l G \longrightarrow \call(\cale \rtimes_l G)$
by
\begin{equation}    \label{ThetaAction}
\Theta_l( \phi_l(a_g g)) = (a_g \otimes 1) (U_g \otimes V_g),
\end{equation}
where $a_g \in A_g, g \in G^{*}$.
It is induced by the $G$-covariant representation $a \mapsto a \otimes 1$ and $g \mapsto U_g \otimes V_g$
(Lemma
\ref{lemmaCovRepInducedMap}),
because $U_g \otimes V_g$ is partial isometry in the $C^*$-algebra $\call(\cale \rtimes_l G) \subseteq B(\calh)$
($\calh$ a Hilbert space).
When $l=i$ then $\Theta_l$ is also well defined as $g \mapsto U_g \otimes V_g$ factors through $\overline G$
(see (\ref{descentVaction2})).
For the $H$-equivariance of $\Theta$ we compute
\begin{equation}
U_h \otimes V_h \,\,\Theta (\phi_l(a_g g))\,\, U_h^* \otimes V_h^* \,=\, \Theta\big(\phi_l\big(h(a_g) g \big)\big)\,\, U_h U_h^* \otimes V_h V_h^*.
\end{equation}
\end{proof}

\section{Hilbert bimodules over reduced crossed products}

The discussion in this section is only related
to the reduced crossed product, that is, when $l = r$. Recall that in this case
we only allow $B=\C$ with the trivial $G$-action.
(Nevertheless we shall write $B$ rather than $\C$ in this section.)
Consequently, the operator $U_g$ ($g \in G$) on a $B$-Hilbert module $\cale$ is
adjoint-able by (\ref{hilbmodrel3}). For the boundedness of the action
of $A \rtimes_r G$ on $\cale \rtimes_r G$ in Proposition
\ref{lemmaThetaReduced} below we will need a standard intertwining
trick for covariant representations tensored by the left regular
representation, see for instance \cite{1097.46042}, Appendix A,
Lemma A.18.(ii).

Let $\cale \otimes \ell^2(G)$ be the skew tensor product of $G$-Hilbert modules.
By Lemma \ref{lemmaTensorSkewwithInterior}
there is an isomorphism
\begin{equation}  \label{congEBell2}
\cale \otimes \ell^2(G) \cong (\cale \otimes_B B) \otimes (\C \otimes_\C \ell^2(G))
\cong \cale \otimes_B \big (B \otimes \ell^2(G) \big ).
\end{equation}
Define a partial isometry $W$ on $\cale \otimes \ell^2(G)$
by
$$W(x_t \otimes e_t) = U_t(x_t) \otimes e_{t}$$
for all $t \in G$ and $x_t \in \cale$ (Lemma \ref{characterizationlemmaPartIso}).
Let
\begin{equation}   \label{MapGamma}
\Gamma:A \rtimes_{\rm alg} G \longrightarrow \call(\cale \otimes \ell^2(G))
\end{equation}
be induced by the covariant representation
\begin{equation}   \label{reprGamma}
\Gamma(a)= (a \otimes 1), \qquad \Gamma(g) = U_g \otimes \lambda_g
\end{equation}
for all $a \in A,g \in G$.
Recall that we write
$$A \rtimes_\Gamma G = \overline{\Gamma(A \rtimes_{\rm alg} G)}.$$

\begin{lemma}  \label{lemmaWWcommutes}
$W W^*$ commutes with the $G$-action $U \otimes V$, with $A \otimes 1$ and with $A \rtimes_\Gamma G$.
\end{lemma}

\begin{proof}
One checks that the projection $W W^*$ commutes with the adjoint-able partial isometry $U_g \otimes \lambda_g$ (and so with $U_g^* \otimes \lambda_g^*$) and $a \otimes 1$ for all $g \in G$ and $a \in A$.
(One uses $U_{\rho_g(t)} U_{\rho_g(t)}^* U_g = U_{g t t^* g^* g} = U_{g t (g^* g t)^*}= U_{g t t^*}$ by transferred left cancellation and
Lemma \ref{lemmaTransferredILM}.) 
\end{proof}

\begin{definition}  \label{definitionGnondegenerate}
{\rm
$G$ is called {\em non-degenerate} if for all Hilbert $(A,B)$-bimodules
and all $x \in A \rtimes_\Gamma G$,
$x W W^* = 0$ implies $x=0$.
}
\end{definition}

If $G$ is a groupoid then
$W W^*$ is an identity for $A \rtimes_\Gamma G$ and so $G$ is non-degenerate.
Indeed, every $y \in \Gamma(A \rtimes G)$ can be written as a product of elements of the form
$x=(a_g \otimes 1) (U_g \otimes \lambda_g) \in A \rtimes_\Gamma G$ for $g \in G'$.
Let $\eta:=\xi_t \otimes e_t \in \cale \otimes \ell^2(G)$. Then
$$x W W^* \eta = a_g U_g U_t U_t^* \xi_t \otimes \lambda_g e_t = a_g U_g \xi_t \otimes \lambda_g e_t = x \eta$$
by Lemma \ref{corollaryHilbertactionInverses}.

Our motivating examples for reduced crossed products were semimultiplicative sets like directed graphs.
A prototype-example is $G=\N_0$.
By showing in the next lemma that $\N_0$ is non-degenerate we would like to demonstrate that non-degenerateness
may not be a too restrictive condition.

\begin{lemma}
$\N_0$ is non-degenerate.
\end{lemma}

\begin{proof}
Let $S$ denote the $\N_0$-action on a Hilbert module $\cale$ with transferred left cancellation.
We claim that every word $S_g$ for $g \in \N_0^*$ allows a representation as $S_g = S_n S_k^* = S_1^n (S_1^k)^*$ for $n,k \in \N_0$.
Indeed, $S_0$ is a unit for every word, as in particular $S_0$ is self-adjoint by Lemma \ref{IdemLemma}.
Also, $S_0 = S_1^* S_1 S_0 = S_1^* S_1$ by transferred left cancellation.
The claim then follows by induction on the length of a word.

Let $X \subseteq A \rtimes_{\rm alg} G \subseteq \F(G,A)$ denote the set of elements of the form $a = \sum_{n,k \in \N_0} a_{n,k} n k^*$
for $a_{n,k} \in A$ (recall identity (\ref{ginfa}) which holds in $\F(G,A)$).
By the above claim, $\Gamma(X) = \Gamma(A \rtimes_{\rm alg} G)$.
Write $p=W W^*$.
To check Definition \ref{definitionGnondegenerate}, assume that $T \in A \rtimes_\Gamma G$ satisfies $T p = 0$.
Then there is a sequence $T^i = \sum_{n,k \in \N_0} a_{n,k}^i n k^*$ in $X$ such that $\Gamma(T^i)$ converges in norm to $T$.

In $\cale \otimes \ell^2(\N_0)$ and by (\ref{reprGamma}) we have
\begin{eqnarray}  
&& \Gamma(T^i)(x_0 \otimes e_0)= \sum_{n,k \in \N_0} a_{n,k}^i S_{n k^*} (x_0) \otimes  \lambda_{n k^*}(e_0) \nonumber \\
&=& \sum_{n \in \N_0} a_{n,0}^i S_n x_0 \otimes e_n = \Gamma(T^i) p (x_0 \otimes e_0) \longrightarrow T p (x_0 \otimes e_0)=0
\label{someequ1}
\end{eqnarray}
when $i \longrightarrow \infty$, since $T p =0$, for all $x_0 \in \cale$.
Similarly we have
\begin{eqnarray}
\Gamma(T^i)(x_1 \otimes e_1) & =& \sum_{n \in \N_0} a_{n,0}^i S_n  x_1 \otimes e_{n+1}
+ \sum_{n \in \N_0} a_{n,1}^i S_n (S_1^*  x_1) \otimes e_n,   \label{someeq7} \\
\Gamma(T^i)p(x_1 \otimes e_1) &=&
(1 \otimes \lambda)\sum_{n \in \N_0} a_{n,0}^i S_n (S_1 S_1^* x_1) \otimes e_{n} \label{someeq8}\\
&&+ \sum_{n \in \N_0} a_{n,1}^i S_n (S_1^*  x_1) \otimes e_n
 \longrightarrow 0    \label{someeq9}
\end{eqnarray}
The convergence is here because of $T p =0$.
Entering convergence (\ref{someequ1}) in convergence (\ref{someeq8})-(\ref{someeq9}) shows that
(\ref{someeq7}) converges to zero (using convergence (\ref{someequ1}) again).
One can proceed in this way further by considering $\Gamma(T_i)(x_2 \otimes e_2)$ and showing that it converges to zero, and so on.
In this way we get $T(x)= \lim_{i \rightarrow \infty}\Gamma(T_i)(x) = 0$ for all $x \in \cale \odot \ell^2(\N_0)$.
Hence $T=0$.
\end{proof}

We now come to the result this section is all about.

\begin{proposition} \label{lemmaThetaReduced}
$\cale \rtimes_r G$ is a $H$-Hilbert $(A \rtimes_r G,B \rtimes_r G)$-bimodule.
\end{proposition}

\begin{proof}
We want to define the action $\Theta_r$ of $A \rtimes_r G$ on $\cale \rtimes_r G$ as in (\ref{ThetaAction}).
Thus we aim to define $\Theta_r$ on $\phi_r(A \rtimes_{\rm alg} G)$ by $\Theta_r \phi_r = \varphi$,
where
$\varphi: A \rtimes_{\rm alg} G \longrightarrow \call(\cale \rtimes_r G)$ is determined by
$$\varphi( a_g g) = (a_g \otimes 1) (U_g \otimes V_g) .$$
We have a commutative diagram
$$\begin{xy}
\xymatrix{
A \rtimes_{\rm alg} G \ar[r]^\varphi  \ar[rd]_\Gamma  &  \call \big (\cale \otimes_B (B \rtimes_r G) ) \ar[rr]^\mu  \ar[d]^f & &
\call \big ( \cale \otimes_B (B \rtimes_r G) \otimes_{B \rtimes_r G} (B \otimes \ell^2(G) \big)
\ar[d]^{\mu_1} \\
& \call \big (\cale \otimes \ell^2(G) \big ) 
& &  \call \big (\cale \otimes_B  (B \otimes \ell^2(G) \big )  \ar[ll]^{\mu_2}
}
\end{xy}$$
Here, $B \rtimes_r G$ acts on $B \otimes \ell^2(G)$ by $\zeta$ of Lemma \ref{lemmaLeftRegActionHilbertModule},
$\mu$ is the injective map of Lemma \ref{lemmaActingOnSpaceTensor}, $\mu_1$ the isomorphism induced by the isomorphism of Lemma \ref{lemmakuerzetensor}, 
and $\mu_2$ the isomorphism induced by the isomorphism (\ref{congEBell2}).
It is important here that $G$ acts trivially on $B$. Hence, in the right bottom corner of the above diagram, $B$ acts on $B \otimes \ell^2(G)$
by left multipliciation (so acts only on $B$).
Let $f:= \mu_2 \mu_1 \mu$, which is injective.
A tedious computation (similar to that of Lemma \ref{lemmaLeftRegActionHilbertModule}) yields
$$f \big( \varphi(a_g g) \big) (x_t \otimes e_t) = a_g U_g x_t \otimes \lambda_g e_{t} = \Gamma(a_g g)$$
for $g \in G^*,t \in G,x_t \in \cale$ and $a_g \in A_g$.
Hence $f \varphi = \Gamma$ on $A \rtimes_{\rm alg} G$.

In order that $\Theta_r$ is evidently a well defined continuous map we need to show that
$$\|\Theta_r ( \phi_r(x)) \| = \|\varphi(x) \| = \|f (\varphi(x)) \| = \|\Gamma(x)\|
\le \|\phi_r(x)\|_{A \rtimes_r G}$$
for all $x \in A \rtimes_{\rm alg} G$. Only the last inequality needs a discussion; the other identites are clear.

Since $G$ is non-degenerate (Definition \ref{definitionGnondegenerate}), the homomorphism
$$\nu : A \rtimes_\Gamma G \longrightarrow ( A \rtimes_\Gamma G) W W^*$$
given by
$\nu(x) = x W W^*$ (see Lemma \ref{lemmaWWcommutes}) is an
isometry.
Thus $\|W W^* \Gamma(x)\| = \|\Gamma(x)\|$ for all $x \in A \rtimes_{\rm alg} G$.

By Lemma \ref{lemmaTransferredILM} and the fact that $U$ has transferred left cancellation, we thus have
\begin{eqnarray*}
 && \Gamma(a_g g) W W^* (\xi_t \otimes e_t) \,= \,a_g U_g U_t U_t^* \xi_t \otimes \lambda_g(e_{t}) \,=\, a_g U_{\rho_g(t)} U_t^* \xi_t \otimes e_{\rho_g(t)}\\
&=&  U_{\rho_g(t)} U_{\rho_g(t)}^* a_g U_{\rho_g(t)} U_t^* \xi_t \otimes e_{\rho_g(t)}
\,=\, U_{\rho_g(t)} \big ((\rho_g(t))^*(a_g) \big ) U_t^* \xi_t \otimes  e_{\rho_g(t)}\\
&=& \big (W \phi_r(a_g g)  W^* \big ) (\xi_t \otimes e_t)
\end{eqnarray*}
for $t \in G, g \in G^*, a_g \in A_g$ and $\xi_t \in \cale$, and when $\rho_g(t)$ is defined.
(Note that $\cale$ is actually a Hilbert space.)
This thus shows
$$\|\Gamma(x)\| = \|\Gamma(x) W W^*\| = \|W \phi_r(x) W^*\| \le \|\phi_r(x)\|.$$
\end{proof}

\section{The descent homomorphism}     \label{sectionDescent3}



Let $B_1$ and $B_2$ be $H \times G$-Hilbert modules.
Let $(\cale_1,T_1) \in \E^G(A,B_1)$ and $(\cale_2,T_2) \in \E^G(B_1,B_2)$.
Write $\cale_{12} = \cale_1 \otimes_{B_1} \cale_2$.

%
%

\begin{lemma} \label{lemmaIsomorphE12}
There is an $H$-Hilbert module isomorphism
\begin{eqnarray*}
\cale_{12} \rtimes_l G    &\cong&  (\cale_1 \rtimes_l G) \otimes_{B_1 \rtimes_l G} (\cale_2 \rtimes_l G).
\end{eqnarray*}
\end{lemma}

\begin{proof}
In the category of $H$-Hilbert modules $B_2 \rtimes_l G$ and $B_2 \rtimes_l^{\rm Mod} G$ are identic, as they differ only in their $G$-action (see Lemma \ref{lemmaHHilbertCstar}).
The map $\varphi: B_1 \longrightarrow B_1 \rtimes_l G$ given by $\varphi(b)= b 1_{G}$
is a $H$-equivariant homomorphism of $H$-Hilbert $C^*$-algebras (Definition \ref{defGequiHilbertCstar}).
By
\cite[Lemma 14]{burgiSemimultiKK}
there is an isomorphism of $H$-Hilbert mdoules
\begin{eqnarray*}
\cale_1 \otimes_{B_1} (B_1 \rtimes_l G) \otimes_{B_1 \rtimes_l G}
\big ( \cale_2 \otimes_{B_2} (B_2 \rtimes_l G) \big)
&\cong& \cale_1 \otimes_{B_1} \cale_2 \otimes_{B_2}  (B_2 \rtimes_l G).
\end{eqnarray*}
\end{proof}

\begin{lemma} \label{lemmaKasparovRemark210}
If $(\cale_{12},T_{12})$ is a Kasparov product then
$R= [T_1 \otimes 1, T_{12}]$ belongs to $Q_{A}(\cale_{12})$, further $R \ge 0$ modulo $I_{A} (\cale_{12})$,
and the elements
\begin{eqnarray*}
g(R) - g(1) R  &=& U_gR U_g^* - U_g U_g^* R,  \label{GandR1}\\
g(1) R - R g(1) &=& U_g U_g^* R  - R U_g U_g^* \label{GandR2}
\end{eqnarray*}
are in $I_{A}(\cale_{12})$ for all $g \in G'$.
\end{lemma}

\begin{proof}
The first two assertions follows from Remark below Definition 2.10 in \cite{kasparov1988},
applied to the trivial group $G= \{e\}$.
Let $a \in A, a'= g^*(a)$ and $T_1' = T_1 \otimes 1$.
For simplicity we compute only the case when $\partial a = 0$.
Modulo $\calk(\cale_{12})$ we have
\begin{eqnarray*}
&&a g(T_{12} T_1') = a g( g^*(1) T_{12} T_1') = g( a' g^*(1) T_{12} T_1') \equiv g (a' T_{12} g^*(1) T_1')\\
&=& a g(T_{12}) g(T_1') \equiv a T_{12} g(1) g(T_1') \equiv T_{12} a g(T_1')\\
&=& T_{12} (k \otimes g(1)) + T_{12} T_1' a g(1),
\end{eqnarray*}
where $k= a g( T_1) - T_1 g(1) a \in \calk(\cale_1)$. Similarly we compute
$$a g(T_1' T_{12}) =  (k \otimes g(1)) T_{12} + T_1' T_{12} a g(1).$$
Hence
$$a g([T_1',T_{12}]) -  [T_1',T_{12}] g(1) a  \equiv [k \otimes g(1), T_{12}] \equiv 0$$
by \cite[Lemma 10.(1)]{burgiSemimultiKK}. Also one has $[ a,[T_1',T]]\equiv 0$ by this lemma.
A similar computation yields the last claim.
\end{proof}

The following lemma is a standard result for crossed products.

\begin{lemma}   \label{lemmacrossedtensor}
If $D$ is a $C^*$-algebra with trivial $G$-action then
$(A \otimes_{\rm max} D) \rtimes G \cong (A \rtimes G) \otimes_{\rm max} D$
(also for the strong crossed product)
and
$(A \otimes_{\rm min} D) \rtimes_r G \cong (A \rtimes_r G) \otimes_{\rm min} D$
canonically.
\end{lemma}

\begin{theorem}  \label{theoremMain}
Let $A$ and $B$ be $H \times G$-Hilbert $C^*$-algebras and $l \in \{\emptyset,s,r,i\}$.
Assume that $G$ is unital.
For all appearing $G \times H$-actions on Hilbert modules and $C^*$-algebras
we require that the induced $H^*$-actions and $G^*$-actions commute.
If $l=r$ then we assume that $G$ is non-degenerate and associative and has left cancellation,
all $G$-Hilbert modules and $G$-Hilbert $C^*$-algebras have transferred left cancellation,
and $B= \C$ with the trivial $G$-action.
Then there exists a descent homomorphism
\begin{eqnarray*}
&& j^G_l : KK^{H \times G} (A,B) \longrightarrow KK^H (A \rtimes_l G, B \rtimes_l G)
\end{eqnarray*}
given by
$$j^G_l(\cale,T) = (\cale \rtimes_l G, T \otimes 1)$$
for all $(\cale,T) \in \E^{H \times G} (A,B)$.
Moreover, the following two points hold true:

(a) If $x_1 \in KK^{H \times G}(A,B_1)$, $x_2 \in KK^{H \times G}(B_1,B_2)$
and the intersection product $x_1 \otimes_{B_1} x_2$ exists
then
$$j^G_l(x_1 \otimes_{B_1} x_2) = j^G_l(x_1) \otimes_{B_1 \rtimes_l G} j^G_l(x_2).$$

(b) If $A=B$ is $\sigma$-unital then $j^G_l(1_A)= 1_{A \rtimes_l G}$.
\end{theorem}

\begin{proof}
In our proof we essentially follow Kasparov \cite{kasparov1988}.
We define
compact operators
$\theta_{\xi,\eta} \in \calk(\calf)$ by $\theta_{\xi,\eta}(x) = \xi
\langle \eta,x\rangle$, where $\xi,\eta,x \in \calf$ and $\calf$ is any Hilbert module. 
Write $Z$ for the diagonal $G$-Hilbert action $U \otimes V$ on $\cale \otimes_B (B \rtimes_l^{\rm Mod} G)$.
Let $\phi_l = \phi_{B,G,l}$.
Let $(a_i)$ be an approximate unit in $B$.
Let $E \in \cale$ and $F \in B \rtimes_l G$. Let $x,y \in G^*$.
Then one has (in $\cale \otimes_B (B \rtimes_l^{\rm Mod} G)$)
\begin{eqnarray*}
&&\theta_{U_{x y^*} (\xi) \otimes \phi_l(xy^{*}(a_i) x )\, ,\, \, \eta \otimes \phi_l (y y^{*}(a_i) y )} (E \otimes F)\\
&=& U_{x y^*} (\xi) \otimes \phi_l \big(xy^{*}(a_i) x \big) \,\,\big \langle \eta \otimes \phi_l \big(y y^{*}(a_i) y \big ), E \otimes F \big \rangle\\
&=& U_{x y^*} (\xi) \otimes \phi_l \big ( xy^{*}(a_i) x \big) \,\,{\phi_l \big(y y^{*}(a_i) y \big)}^* \,\,\phi_l \big(\langle \eta,E\rangle \big)\, F\\
&=& U_{x y^*} (\xi) \otimes \phi_l \big( xy^{*}(a_i) x  \,\, y^* y y^{*}(a_i^*) y^* \,\,\langle \eta,E\rangle \big) F\\
&=& U_{x y^*} (\xi) \otimes \phi_l \big( xy^{*}(a_i)\,\, x y^{*}(a_i^*) \,\, x y^*(\langle \eta,E\rangle) \big)  \, \phi_l(x y^*) F\\
&=& U_{x y^*} \big(\xi \, a_i\,\, a_i^* \, \langle \eta,E\rangle \big)  \otimes \phi_l(x y^*) F\\
&=& U_{x y^*} \otimes V_{x y^*} \,\,\, \big (\theta_{\xi a_i a_i^*,\eta} \otimes 1\,\,\, \big(E \otimes F \big) \big ).
\end{eqnarray*}
Omitting here $E \otimes F$ and then taking the limit $i \rightarrow \infty$ yields
$$Z_{x y^*} \big(\calk(\cale) \otimes 1 \big) \subseteq \calk \big(\cale \otimes_B(B \rtimes_l^{\rm Mod} G) \big).$$
For $x \in G'$ we have $Z_x= Z_x Z_x^* Z_x$, and since $Z_x (\calk) \subseteq \calk$, we obtain
\begin{equation}   \label{equZcompact}
Z_{x} \big(\calk(\cale) \otimes 1 \big) \subseteq \calk \big(\cale \otimes_B(B \rtimes_l^{\rm Mod} G) \big).
\end{equation}

Let $\Theta$ be the action of $A \rtimes_l G$ on $\cale \rtimes_l G$, see (\ref{ThetaAction}).
By (\ref{equZcompact})
it is straight forward to compute that
$$[\Theta\big(\phi_{l}(a_g g) \big) , T \otimes 1]
\in \calk(\cale \rtimes_l G)$$
for all $g \in G'$, where $\phi_l$ denotes $\phi_{A,G,l}$ (use $a U_g = U_g U_g^* a U_g = U_g g(a)$). This result extends by induction to all $g$ in $G^{*}$ by using products:
write
$\Theta \big(\phi_l(a g h) \big)$ as
$$\Theta \big(\phi_l(a g h) \big) = \Theta \big (\phi_l(a^{1/2} g) \big) \Theta \big(\phi_l(g^{*}(a^{1/2}) h) \big)$$
for $g \in G^{*},h \in G'$ and positive $a \in A_{gh}$ by (\ref{ginfa}) and Lemma \ref{lemmaHilbertCstarGstar} (iii).
By similar computations one easily checks all other requirements showing that $(\cale \rtimes_l G, T \otimes 1)$
is a cycle.

The map $j^G$ is well defined, as a homotopy $(\calf,S) \in \E^{H \times G} \big(A,B[0,1] \big)$
gives a homotopy $j^G(\calf,S) \in \E^G \big(A \rtimes_l G, B[0,1] \rtimes_l G \big)$, as
\begin{eqnarray*}
B[0,1] \rtimes_l G &\cong& \big( B \rtimes_l G \big)  \otimes C[0,1],\\
\calf \otimes_{B[0,1]} \big(B[0,1] \rtimes_l G \big) \otimes_{B[0,1] \rtimes_l G} \big (B \rtimes_l G \big)  &\cong&
\calf_t \otimes_B \big ( B \rtimes_l G \big )
\end{eqnarray*}
for $0 \le t \le 1$, where the first isomorphism is by Lemma \ref{lemmacrossedtensor}
and the second isomorphism follows from Lemma \ref{lemmaTensorSkewwithInterior}.

To prove (a), let $x_1 = (\cale_1,T_1)$, $x_2 =(\cale_2,T_2)$,
$\cale_{12} = \cale_1 \otimes_{B_1} \cale_2$
and $(\cale_{12},T_{12})$ a Kasparov product of $x_1$ and $x_2$.
We have to check that $j^G(\cale_{12}, T_{12})=(\cale_{12} \rtimes_l G, T_{12} \otimes 1)$ is a Kasparov product of
$j^G(x_1)=(\cale_1 \rtimes_l G, T_1 \otimes 1)$
and $j^G(x_2)= (\cale_2 \rtimes_l G, T_2 \otimes 1)$.
For the definition of a {Kasparov
product} $(\cale_{12}, T_{12})$ of $(\cale_1,T_1)$ and
$(\cale_2,T_2)$ we shall use \cite[Definition 19]{burgiSemimultiKK} (cf. \cite{skandalis}).
It states that $\cale_{12}=\cale_1 \otimes_{B_1} \cale_2$, $T_1 \otimes 1$ is a $T_2$-connection on $\cale_{12}$,
and $a [T_1 \otimes 1, T_{12}] a^* \ge 0$ in the quotient $\call(\cale_{12})/\calk(\cale_{12})$ for all $a \in A$.
For the definition of a {$T_2$-connection
on $\cale_{12}$} see \cite{skandalis}, or \cite[Definition
2.6]{kasparov1988}, or \cite[Definition 18]{burgiSemimultiKK}.

We use the isomorphism given in Lemma \ref{lemmaIsomorphE12}.
For the $H$-equivariant $*$-homomorphism
\begin{equation}  \label{injectHomCrossUnitG}
f:B_2 \longrightarrow B_2 \rtimes_l G, \quad f(b) = b 1_G,
\end{equation}
$j^G(\cale_{12}, T_{12})= f_*((\cale_{12}, T_{12}))$ is a cycle in $\E^H(A \rtimes_l G, B \rtimes_l G)$
by \cite[Definition 24]{burgiSemimultiKK}.

The $G$-action on $\cale_{12}$ will be denoted by $U$.
The inclusion
\begin{eqnarray*} \label{equInclusKomp}
\calk(\cale_2, \cale_{1} \otimes_{B_1} \cale_2) \otimes 1_{B_2 \rtimes_l G}  &\subseteq &
\calk \big(\cale_2 \otimes_{B_2} (B_2 \rtimes_l G), \,\cale_{1} \otimes_{B_1} \cale_2  \otimes_{B_2} (B_2 \rtimes_l G) \big),
\end{eqnarray*}
where $B_2$ acts by $f$,
is similarly proved as \cite[Lemma 15]{burgiSemimultiKK}.

We use it to check
\begin{eqnarray*}
\theta_\eta (T_2^t \otimes 1) - (-1)^{\partial \eta \cdot \partial T_2} (T_{12}^t \otimes 1) \theta_\eta
&\in&  \calk(\cale_2 \rtimes_l G, \cale_{12} \rtimes_l G)
\end{eqnarray*}
for $\eta \in \cale_1, t \in \{1,*\}$ and
$$\theta_\eta (\xi \otimes z) = \eta \otimes \xi \otimes z$$
for
$\xi \in \cale_2, z \in B_2 \rtimes_l G$.
This shows that
$T_{12} \otimes 1$ is a $T_2 \otimes 1$-connection on $\cale_{12} \rtimes_l G$.

By \cite[Lemma 15]{burgiSemimultiKK} and the homomorphism $f$ we have
\begin{eqnarray} \label{compactsinput}
\calk(\cale_{12}) \otimes 1 &\subseteq& \calk (\cale_{12} \rtimes_l G).
\end{eqnarray}

By Lemma \ref{lemmaKasparovRemark210} we have $R + k \ge 0$ for $R=[T_1 \otimes 1, T_{12}]$ and some $k \in I_A(\cale_{12})$.
Let $a \in A$ (actually $\pi(A) \otimes 1$!), $g \in G'$, and note that $a U_g  = U_g U_g^* a U_g = U_g g^*(a)$
for $a \in A$ and $g \in G'$.
Using inclusion (\ref{compactsinput}),
Lemma \ref{lemmaKasparovRemark210},
and the fact that $U_g \otimes V_g$ is in $\call(\cale_{12} \rtimes_l G)$,
we have
the next computation in $\cale_{12} \rtimes_l G = \cale_{12} \otimes_{B_2} (B \rtimes_l^{\rm Mod} G)$
modulo $\calk (\cale_{12} \rtimes_l G)$ for $g \in G'$.
\begin{eqnarray*}
&& a (U_g  \otimes V_g) (R \otimes 1) = U_g g^*(a)  U_g^* U_g  R\otimes V_g \equiv a U_g R U_g^* U_g \otimes V_g\\
& \equiv&  a R U_g \otimes V_g = a (R \otimes 1) (U_g \otimes V_g).
\end{eqnarray*}
By induction on the length of a word in $G^*$ we see that this identity holds true also for all $g \in G^*$.

Let $a = \sum_g a_g g \in C_c(G,A)$. Let $\phi_l = \phi_{A,G,l}$.
By the last computation we have the following computation
in the quotient $\call(\cale_{12} \rtimes_l G)/\calk (\cale_{12} \rtimes_l G)$, where $\underline R := R + k \ge 0$.
\begin{eqnarray*}
&&(\Theta \otimes 1)(\phi_l(a))\, (R \otimes 1)\,  (\Theta \otimes 1 )(\phi_l(a))^* \\
&=& \Big[ \Theta \otimes 1\, \Big(\phi_l \Big(\sum_{g \in G^*} a_g g \Big) \Big) \Big](R \otimes 1) \Big[\Theta \otimes 1 \,\Big (\phi_l \Big(\sum_{h \in G^*} a_h h \Big ) \Big) \Big]^*\\
&=& \sum_{g,h \in G^*} a_g U_g R U_h^* a_h^* \otimes V_g V_h^*
\,=\, \sum_{g,h \in G^*} U_g g^*(a_g) \underline R U_h^* a_h^* \otimes V_g V_h^*\\
&=& \sum_{g,h \in G^*} a_g \underline R^{1/2} U_g U_h^*  \underline R^{1/2} a_h^* \otimes V_g V_h^* \,\ge \, 0.
\end{eqnarray*}
Note that
\begin{eqnarray*}
R \otimes 1  &=&  [T_1 \otimes 1 \otimes 1, T_{12} \otimes 1].
\end{eqnarray*}
This shows that $(\cale_{12} \rtimes_l G, T_{12} \otimes 1)$ is a Kasparov product.
We have thus checked point (a).

Point (b) follows from $j_l^G(A,0) = (A \otimes_A (A \rtimes_l G),0) = (A \rtimes_l G,0)$ by using a map like in (\ref{injectHomCrossUnitG}).
\end{proof}

\bibliographystyle{plain}
\bibliography{references}

\end{document}